\documentclass[11pt]{amsart}
\usepackage{amscd,amsmath,amssymb,amsfonts, appendix}
\usepackage{color}
\usepackage[cmtip, all]{xy}
\usepackage{graphicx}
\usepackage{hyperref}
\usepackage{enumerate}
\newtheorem{thm}[equation]{Theorem}
\newtheorem{ethm}[equation]{Expected Theorem}

\newtheorem{lemma}[equation]{Lemma}

\newtheorem{proposition}[equation]{Proposition}
\newtheorem{prop}[equation]{Proposition}
\newtheorem{corollary}[equation]{Corollary}
\newtheorem{definition}[equation]{Definition}

\newtheorem{cor}[equation]{Corollary}

\newtheorem{question}[equation]{Question}
\newtheorem{defn}[equation]{Definition}
\newtheorem{conj}[equation]{Conjecture}

\newtheorem*{thm*}{Theorem}
\newtheorem{remark}[equation]{Remark}

\numberwithin{equation}{section}

\newcommand{\isoarrow}{{~\overset\sim\longrightarrow~}}

\newcommand{\CO}{{\mathcal{O}}}

\newcommand{\hG}{{\hat{G}}}

\newcommand{\ZZ}{{\mathbb Z}}

\newcommand{\Gm}{{\mathbb G}_m}

\newcommand{\CG}{{\mathcal{G}}}

\newcommand{\CA}{{\mathcal{A}}}
\newcommand{\CT}{{\mathcal{T}}}
\newcommand{\CL}{{\mathcal{L}}}

\newcommand{\CV}{{\mathcal{V}}}

\newcommand{\GG}{{\mathbb G}}

\newcommand{\ra}{{~\rightarrow~}}

\newcommand{\QQ}{{\mathbb Q}}

\newcommand{\Ql}{{\mathbb Q}_{\ell}}

\newcommand{\Qlb}{{\overline{\mathbb Q}_{\ell}}}

\newcommand{\Fp}{{{\mathbb F}_p}}
\newcommand{\Fq}{{{\mathbb F}_q}}
\newcommand{\RR}{{\mathbb R}}
\newcommand{\ad}{{\mathbf A}}

\newcommand{\CC}{{\mathbb C}}
\newcommand{\PP}{{\mathbb P}}

\newcommand{\Qp}{{{\mathbb Q}_p}}

\begin{document}

\author{Wee Teck Gan}
\thanks{W.T. Gan was partially supported by a Singapore government MOE Tier 1 grant R-146-000-320-114.
}

\address{Wee Teck Gan\\Mathematics Department, National University of Singapore  \\Block S17, 10 Lower Kent Ridge Road
Singapore 119076
}
 \email{matgwt@nus.edu.sg}

\author{Michael Harris}
\thanks{M. Harris  was partially supported by NSF Grant DMS-2001369, and by a Simons Foundation Fellowship, Award number  663678
}

\address{Michael Harris\\
Department of Mathematics, Columbia University, New York, NY  10027, USA}
 \email{harris@math.columbia.edu}

 \author{Will Sawin}
\thanks{W. Sawin was partially supported by NSF grant DMS-2101491 and served as a Clay Research Fellow.
}

\address{Will Sawin\\
Department of Mathematics, Columbia University, New York, NY  10027, USA}
 \email{sawin@math.columbia.edu}

\date{\today}

\title[Local parameters of supercuspidal representations]{Local parameters of supercuspidal representations}

\maketitle

\begin{abstract}
For a connected reductive group $G$ over a non-archime\-dean local field $F$ of positive characteristic,   Genestier and Lafforgue have attached  a semisimple parameter $\CL^{ss}(\pi)$ to each irreducible representation $\pi$. Our first result shows that the Genestier-Lafforgue parameter of a tempered $\pi$  can be uniquely refined to a tempered L-parameter $\CL(\pi)$, thus giving the unique local Langlands correspondence which is compatible with the Genestier-Lafforgue construction. 
Our second result establishes ramification properties of $\CL^{ss}(\pi)$ for unramfied $G$ and supercuspidal $\pi$ constructed by induction from an open compact (modulo center) subgroup.  If $L^{ss}(\pi)$ is pure in an appropriate sense, we show that  $\CL^{ss}(\pi)$ is ramified (unless $G$ is a torus).   If the inducing subgroup is sufficiently small in a precise sense, we show $\mathcal{L}^{ss}(\pi)$ is wildly ramified.  The proofs are  via global arguments, involving the construction of Poincar\'e series with strict control on ramification when the base curve is $\PP^1$ and a simple application of Deligne's Weil II.
\end{abstract}

\tableofcontents

%\Part{Outline}

\section{Introduction}

Let $G$ be a connected reductive group over a non-archimedean local field $F$ and let $\pi$ be an irreducible smooth representation of $G(F)$. The local Langlands conjecture (LLC) posits that to such a $\pi$ can be attached an L-parameter 
\[  \CL(\pi) : WD_F \ra {}^LG(C), \]
which is  an admissible homomorphism from the Weil-Deligne group  $WD_F = W_F \times SL_2(C)$ with values in the Langlands $L$-group ${}^LG$ over an appropriate algebraically closed field $C$ of characteristic $0$. In what follows, we take $C = \Qlb$ with $\ell$ different from ${\rm char}(F)$. One condition of being admissible is that $\CL(\pi)$ should be {\it semisimple}, in the sense that  if $\CL(\pi)(WD_F) \cap \hat{G}(C)$ (where $\hat{G}(C)$ is the Langlands dual group of $G$) is contained in a parabolic subgroup $P \subset \hat{G}(C)$, then it is contained in a Levi subgroup of $P$.  Such L-parameters can also be described in the language of semisimple Weil-Deligne parameters $(\rho, N)$, as we recall in \S \ref{pureWD}.
  \vskip 5pt

 On the other hand,  when $F$ is of positive characteristic $p$,   Genestier and Lafforgue \cite{GLa} have defined a {\it semisimple} Weil parameter
\[ \CL^{ss}(\pi): W_F \ra {}^LG(C). \]
This semisimple parameter $\CL^{ss}(\pi)$ should be related to $\CL(\pi)$ via
\begin{equation} \label{E:ssL}
  \CL^{ss}(\pi) (w)  =  \CL(\pi) \left( w, \left( \begin{array}{cc}
|w|^{1/2} & \\
& |w|^{-1/2} \end{array} \right) \right). \end{equation}
A natural question is then: can the semisimple parameter $\CL^{ss}(\pi)$ of Genestier-Lafforgue be enriched to give the true L-parameter $\CL(\pi)$?
\vskip 5pt

By the properties of the Genestier-Lafforgue construction (recalled in \S \ref{llparam}), it suffices to address this question for (essentially) discrete series representations.
For such a $\pi$, one desiderata of  the LLC is that $\CL(\pi)$ should be {\it irreducible} (or {\it elliptic}) -- its image is contained in no proper parabolic $P$, and in particular is essentially tempered (or equivalently pure). It is not hard to see (cf. Lemma \ref{L:purecompletion})  that for any semisimple parameter $\CL^{ss}(\pi)$, there is at most one essentially tempered L-parameter that could give rise to it via (\ref{E:ssL}). The first result of this paper is the following theorem:
\vskip 5pt

\begin{thm} \label{T:Lparameter}
Let $G$ be a connected reductive group over a local function field $F$ of positive characteristic $p > 3$ (for simplicity). 
Let $\pi$ be an irreducible discrete series representation of $G(F)$. Then there exists a (necessarily unique) essentially tempered L-parameter $\CL(\pi)$ whose associated semisimple parameter (via (\ref{E:ssL})) is the Genestier-Lafforgue parameter $\CL^{ss}(\pi)$. 
\end{thm}
In particular, for a discrete series $\pi$, we have a unique candidate $\CL(\pi)$ for its L-parameter. As a consequence, we have:

\vskip 5pt

\begin{corollary}  \label{C:Lparameter}
There is a well-defined 
\[ 
 {\CL}:  \mathcal{T}(G(F)) \longrightarrow  \{ \text{essentially tempered L-parameters} \}  \]
refining the construction of Genestier-Lafforgue, where $\mathcal{T}(G(F))$ is the set of equivalence classes of irreducible essentially tempered representations of $G(F)$.
\end{corollary}
Unfortunately,  we are not quite able to show that the image of a discrete series representation under $\CL$   is  irreducible (or elliptic).
\vskip 10pt

Assume now that $G$ is an unramified reductive group.
Our second result concerns the ramification property of the semisimple parameter $\CL^{ss}(\pi)$ of a supercuspidal representation $\pi$.  
If the Frobenius eigenvalues of $\CL^{ss}(\pi)$ are pure in an appropriate sense, then the LLC asserts that   $\CL^{ss}(\pi)$ itself is in fact {\it irreducible}.  In particular, unless $G$ is a torus, one expects that $\CL^{ss}(\pi)$ is {\it ramified}.  More precisely, we show:

\vskip 10pt

\begin{thm}\label{mainthm}  
Suppose $F$ is the non-archimedean local field $k(\!(t)\!)$, where $k = \Fq$ is a finite field of order $q = p^m > 5$ for some prime $p$.    Let $G$ be an unramified connected reductive  group over $F$ which is not a torus.  Assume that  $\pi$ is an irreducible  supercuspidal representation of $G(F)$ of the form
$$\pi \isoarrow {\rm c-Ind}_U^G(F) \tau$$
for some (irreducible) smooth representation $\tau$ of a compact open (modulo center) subgroup  $U$ with coefficients in $\Qlb$. 
 
 \vskip 5pt

\noindent (i)  Suppose that $\CL^{ss}(\pi)$ is {\rm pure}:  if ${\rm Frob}_F \in W_F$ is any Frobenius element, then the eigenvalues of $\CL^{ss}(\pi)({\rm Frob}_F)$ are all Weil $q$-numbers of the same weight  (in which case we say that $\pi$ is a {\bf pure supercuspidal} representation). 
Then $\CL^{ss}(\pi) = \CL(\pi)$ is ramified:  it is non-trivial on the inertia subgroup $I_F \subset W_F$. 
 \vskip 5pt
 
 \noindent (ii)  Suppose that $U$ is sufficiently small. Then $\CL^{ss}(\pi)$ (and hence $\CL(\pi)$) is wildly ramified: it is non-trivial on the wild inertia subgroup of $W_F$
 \end{thm}
 Let us make a few  remarks:
 \vskip 5pt
 \begin{itemize}
 \item Conjecturally, every $\pi$ is compactly induced in the above sense, and this property was
recently proved by Fintzen to be true as long as $p$ does not divide  the order of the Weyl group of $G$. 

 \item for (i) to be non-vacuous, it would have been good to produce some examples of supercuspidal representations whose Genestier-Lafforgue parameter $\CL^{ss}(\pi)$ is pure. This does not seem so simple to do. In \S \ref{S:questions}, we produce some such examples  of depth 0 generic supercuspidal representations.
  
  \item  The precise definition of ``sufficiently small" in (ii) above is given at the beginning of \S \ref{wildram}.  \end{itemize}

\vskip 10pt

   When $G = {\rm GL}(n)$,  our two theorems are part of the local Langlands correspondence (LLC) proved by Laumon, Rapoport and Stuhler \cite{LRS93}.  When $G$ is a split classical group,  the results were proved in most cases by Ganapathy and Varma by comparison with Arthur's LLC over $p$-adic fields, via the method of close local fields \cite{GV,A}.  Thus the novelty of this paper is that, like the Genestier-Lafforgue result, Theorems \ref{T:Lparameter} and  \ref{mainthm} are valid uniformly for all (unramified) reductive groups, including exceptional groups.   
\vskip 5pt

Of course, Theorem \ref{mainthm}(i) is a much weaker statement than the expected local Langlands correspondence, which not only asserts that $\CL^{ss}(\pi)$ is irreducible when $\pi$ is pure but gives a complete description of the {\it $L$-packet} of representations with the given local parameter.  However, all proofs of the local Langlands correspondence for ${\rm GL}(n)$ begin with the observation that a supercuspidal representation $\pi$ of ${\rm GL}(n,F)$ is not {\it incorrigible}  \cite{H19} -- that it becomes a constituent of a principal series representation after a finite series of base changes corresponding to cyclic extensions $F_{i+1}/F_i$, $i = 0, \dots, n$, where $F_0 = F$.  Since one can choose the $F_i$ so that the Galois parameter restricted to $W_{F_{n+1}}$ is unramified, this is obvious if we know that 
\begin{itemize}
\item[(1)] no supercuspidal representation of any Levi factor of ${\rm GL}(n,F)$ other than ${\rm GL}(1)^n$ has an unramified parameter, and
\item[(2)] the parametrization is compatible with cyclic stable base change.
\end{itemize}
In the proofs of the LLC in \cite{LRS93, HT01,He00}, point (1) is proved by reference to 
Henniart's numerical local correspondence \cite{He88}, whereas point (2) is proved by a global method.  In Scholze's proof of the LLC for ${\rm GL}(n)$ \cite{Sch13}, point (1) is proved by a geometric argument using a study of nearby cycles in an integral model of the Lubin-Tate local moduli space.  
\vskip 5pt

Starting with point (1), the LLC  is deduced by a study of the fibers of (stable) base change for cyclic extensions of
$p$-adic fields.  This is ultimately a consequence of deep properties -- the ``advanced theory" -- of the Arthur-Selberg trace formula, that
are established by Arthur and Clozel in \cite{AC}. For ${\rm GL}(n)$ over a number field, the trace formula and its twisted analogue are automatically stable, but there is  as yet no stable trace formula for general groups over function fields.  If and when a stable trace formula is developed in this generality, it is likely that Theorem \ref{mainthm}(i) will provide the starting point for an inductive proof of the LLC, at least for pure supercuspidal representations.

\vskip 5pt

The proofs of both Theorems \ref{T:Lparameter} and \ref{mainthm} proceed via  global-to-local arguments and thus involve globalizing a given supercuspidal (or discrete series) representation.  For a supercuspidal representation, this globalization is achieved by an adaptation of the method of \cite{GLo} which involves a delicate construction of Poincar\'e series
with precise control on ramification.  Let us elaborate on the proof of Theorem \ref{mainthm}(i) as an example. 
\vskip 5pt

Take  the base curve to be $Y = \PP^1$ and let $K$ denote the global function field $k(Y) = k(t)$.  
By the method of Poincar\'e series used in \cite{GLo}, we construct a cuspidal automorphic representation $\Pi$ of $G(\ad_K)$ that is unramified outside the set of places of $K$ corresponding to the set $Y(k)$ of $k$-rational points of $Y$.  
At the points $v = 0$ and $v = \infty$,  we arrange for the local components of $\Pi$  to be tamely ramified, with $\Pi_\infty^{I_-(p)} \neq 0$ and $\Pi_0^{I^+(p)} \neq 0$; here $I_-$ and $I^+$ denote opposite Iwahori subgroups of 
$G(k(\frac{1}{t}))$ and $G(k(t))$ respectively and $I_-(p) \subset I_-$ and $I^+(p) \subset I^+$ are the corresponding maximal pro-$p$ subgroups.   At the other rational points $v$, we assume $\Pi_v$ to be isomorphic to the fixed supercuspidal $\pi$ in the context of Theorem \ref{mainthm}(i).   Moreover, we assume $\Pi_0$ contains a sufficiently regular character $\alpha$ of the Iwahori subgroup $I^+$.  V. Lafforgue has attached to such a $\Pi$ a local system $L(\pi)$ over $Y \setminus \PP^1(k)$.   Supposing that $\CL^{ss}(\pi)$ is pure and unramified, it follows by Deligne \cite{De80} that $L(\pi)$ extends to a local system over $\Gm$.  Provided $q > 3$,  we can then choose $\alpha$ to obtain a contradiction.
\vskip 5pt

We conclude the introduction with a brief summary of the sections that follow.
After recalling the results of V. Lafforgue and Genestier-Lafforgue in \S \ref{llparam}, and basic properties  about Weil-Deligne parameters in \S \ref{pureWD}, we complete the proof of Theorem \ref{T:Lparameter} at the end of \S \ref{pureWD}.  We then lay the groundwork for the proof of Theorem \ref{mainthm}(i). In \S \ref{S:local}, we establish some preparatory results about local systems on open subsets of $\PP^1$ and in \S \ref{S:unramred}, we prove a useful result (Lemma \ref{parahoric-coordinates-lemma}) about parahoric subgroups of unramified groups.
 In \S \ref{S:open}, we construct some open compact subgroups of $G(\mathbb{A}_K)$ via building theoretic considerations  and use these for the construction of appropriate Poincar\'e series. Using this Poincar\'e series construction and our earlier results about local systems on $\PP^1$, we prove Theorem \ref{mainthm}(i) in \S \ref{S:mainthm}. The proof of Theorem \ref{mainthm}(ii) is given in \S \ref{wildram}: this requires yet another Poincar\'e series construction. We then highlight some natural questions suggested by our results in \S \ref{S:questions}. An example is whether one can show that a generic supercuspidal representation is  pure. We show this for certain generic supercuspidal representations of depth 0.
 In \S \ref{BCIR} we explain the implications of our results for incorrigible representations and prove a weak version of the local Langlands correspondance for $GL(n)$:  all supercuspidal representations of $GL(n,F)$ are either pure with ramified Galois parameters or have wildly ramified parameters.  The purpose of this is to provide a sufficently general template that has the potential of being extended to general $G$. Finally, in \S \ref{DKFS} we propose the obvious conjecture linking the Fargues-Scholze parametrization of representations of groups over $p$-adic fields to the Genestier-Lafforgue parametrization, by means of the Deligne-Kazhdan method of close local fields.

%In \S \ref{wildram} we prove that supercuspidal representations compactly induced from sufficiently small compact mod center subgroups have wildly ramified Galois parameters, provided they are pure.  In \S \ref{pureWD} we consider supercuspidal representations that are not necessarily pure and prove, in nearly all cases, that their semisimple Galois parameters can be extended to Weil-Deligne parameters that satisfy purity of the monodromy weight filtration, in an appropriate sense.  The result does not depend on
%realizing the representation by compact induction.  In this way we can attach full Weil-Deligne parameters to arbitrary supercuspidal representations, and the method extends to general discrete series representations if we admit a version of the stable trace formula over function fields.  In \S \ref{BCIR} we explain the implications of our results for incorrigible representations and prove a weak version of the local Langlands correspondance for $GL(n)$:  that all supercuspidal representations of $GL(n,F)$ are either pure with ramified Galois parameters or have wildly ramified parameters.  Finally, in \S \ref{DKFS} we propose the obvious conjecture linking the Fargues-Scholze parametrization of representations of groups over $p$-adic fields to the Genestier-Lafforgue parametrization, by means of the Deligne-Kazhdan method of close local fields.

\section*{Acknowledgments}  The authors thank H\'el\`ene Esnault, Tony Feng, Jessica Fintzen, Dennis Gaitsgory, Guy Henniart, Tasho Kaletha, Jean-Pierre Labesse, Bertrand Lemaire, and Zhiwei Yun for helpful discussions of various aspects of the paper.  We benefited from exchanges with Laurent Fargues and Peter Scholze regarding the material in the final section.  We are especially grateful to Luis Lomel\'i for his role, in collaboration with one of us, in developing the main techniques on which our results are based, and for his comments in the early stages of this project. Finally, we thank Raphael Beuzart-Plessis for providing us with the appendix, allowing us to prove Theorem \ref{T:Lparameter} for discrete series representations.

\section{Global and local Langlands-Lafforgue parameters}\label{llparam}

We begin by reviewing the global results of V. Lafforgue and the local results of Genestier-Lafforgue.
\vskip 5pt

\subsection{The global results of V. Lafforgue}\label{rev2}
Let  $Y$ be a smooth projective curve over $k$ and $K = k(Y)$  be its global function field.  
Let $G$ be a connected reductive algebraic group over $K$.  Let $\hG$ be the Langlands dual group of $G$ with coefficients in $\Qlb$, and let ${}^LG = \hG \rtimes Gal(K^{sep}/K)$ be the Langlands $L$-group of 
$G$ (in the Galois form).  
\vskip 5pt

Let $\CA_0(G) = \CA_0(G,Y)$ denote the set of cuspidal automorphic representations of $G$ with central character of finite order.  We let 
$\CG^{ss}(G)$ denote the set of equivalence classes of compatible families of semisimple $\ell$-adic homomorphisms, for $\ell \neq p$:
$$\rho_\ell:  Gal(K^{sep}/K) \ra {}^LG.$$
The term {\it semisimple} is understood to mean that if $\rho_\ell(Gal(K^{sep}/K))\cap \hG$ is contained in a parabolic subgroup $P \subset \hG$, then it is contained in a Levi subgroup of $P$.

If $\nu:   G \ra \Gm$ is an algebraic character, with $\Gm$ here designating the split $1$-dimensional torus over $K$, the theory of $L$-groups provides a dual character $\hat{\nu}:  \Gm \ra \hG \subset {}^LG$.  If $Z$ is the center of $G$, and if $c:  \Gm \ra Z \subset G$ is a homomorphism, then the theory of $L$-groups provides an algebraic character ${}^Lc:  {}^LG \ra \Gm$.

We state Lafforgue's theorem for general reductive groups over general global function fields, but we will mainly be
applying it to $K = k(\PP^1)$.  In what follows, a representation of $G$ over a local field $F$ is ``unramified" if it contains a vector invariant under a hyperspecial maximal compact subgroup, in which case the group itself is assumed to be unramified over $F$.
\vskip 5pt

\begin{thm}\label{paramVL} Let $K$ be any global function field and $G$ a connected reductive algebraic group over $K$.
\vskip 5pt

 (i) \cite{Laf18}[Th\'eor\`eme 0.1]    There is a map
$$\CL:  \CA_0(G) \ra \CG^{ss}(G)$$
with the following property:  if $v$ is a place of $K$ and $\Pi \in \CA_0(G)$ is a cuspidal automorphic representation such that $\Pi_v$ is unramified, then $\CL(\Pi)$ is unramified at $v$, and
$\CL^{ss}(\Pi) ~|_{W_{K_v}}$ is the Satake parameter of $\Pi_v$.\footnote{Here and below we will mainly refer to the restriction of a global Galois parameter to the local Weil group, rather than to the local Galois group, because the unramified Langlands correspondence relates spherical representations to unramified homomorphisms of the local Weil group to the $L$-group.   But the difference is inessential.}

(ii)  Suppose  $\nu:   G \ra \Gm$ is an algebraic character and
$$\chi:  \ad_K^{\times}/K^\times  \ra \Qlb^\times = GL(1,\Qlb)$$ is a continuous character of finite order.  For any $\Pi \in  \CG_0(G)$, let $\Pi\otimes \chi \circ \nu$ denote the twist of $\Pi$ by the character $\chi\circ\nu$ of $G(\ad)$. Then 
$$\CL(\Pi\otimes \chi\circ\nu) = \CL(\Pi)\cdot \hat{\nu}\circ \hat{\chi}$$
where $\hat{\nu}$ is as above and where $\hat{\chi}: Gal(K^{sep}/K)^{ab} \ra GL(1,\Qlb)$ is the character corresponding to $\chi$ by local class field theory. 

(iii) \cite{GLa}[Th\'eor\`eme 0.1]  Let  $v$ be a place of of $K$ and let $F = K_v$.  Then the semisimplification of the restriction of $\CL(\Pi)$ to $W_{K_v}$ depends only on $F$ and the local component $\Pi_v$ of $\Pi$, which is an irreducible admissible representation of $G(F)$, and not on the rest of the automorphic representation $\Pi$, nor on the global field $K$.   We denote this parameter $\CL^{ss}(\Pi_v)$.
\end{thm}
\vskip 5pt

\subsection{The local results of Genestier-Lafforgue}
Now let $F$ be a non-archimedean local field of characteristic $p$, so that  $F \isoarrow k'(\!(t)\!)$ for some finite extension $k'$ of $\Fp$.  Let $\CA(G,F)$ denote the set of equivalence classes of  irreducible admissible  representations of $G(F)$, and let 
$\CG^{ss}(G,F)$ denote the set of equivalence classes of semisimple $\ell$-adic homomorphisms
$$\rho:  W_F \ra {}^LG.$$
By (iii) above, we thus obtain a (semisimple) parametrization of $\CA(G,F)$:
\begin{equation}\label{localparam}
\CL^{ss}:  \CA(G,F)   \ra \CG^{ss}(G,F).
\end{equation}
Of course, the theorem quoted above only constructs $\CL^{ss}(\pi)$ when $\pi$ can be realized as the local component $\Pi_v$ of a global cuspidal automorphic representation $\Pi$, but the statement of Th\'eor\`eme 0.1 of \cite{GLa} includes the extension to a (semisimple) parametrization of all members of $\CA(G,F)$.   Theorem \ref{paramVL} continues:

\begin{thm}[\cite{GLa},Th\'eor\`eme 0.1]\label{parab}
(iv)  The local parametrization $\CL^{ss}$ is compatible with parabolic induction in the following sense:  Let $P \subset G$
be an $F$-rational parabolic subgroup with Levi factor $M$ and let $\sigma \in \CA(M,F)$.  If $\pi \in \CA(G,F)$ is an irreducible constituent of  $Ind_{P(F)}^{G(F)} \sigma$ (normalized induction) and $i_M:  {}^LM \ra {}^LG$ is the inclusion of the $L$-group of $M$
as a Levi factor  of a parabolic subgroup in the $L$-group of $G$, then 
$$\CL^{ss}(\pi) = i_M\circ\CL^{ss}(\sigma).$$
\vskip 5pt

(v) In particular, suppose $G$ is quasi-split over $F$ with Borel subgroup $B$, $T \subset B$ is a maximal torus, $I \subset G(F)$ is an Iwahori subgroup, and $I(p) \subset I$ is its maximal pro-$p$ subgroup.  Suppose $\pi^{I(p)} \neq 0$.  Then $\CL^{ss}(\pi)$ takes values in the $L$-group of $T$.
\vskip 5pt

(vi) The local and global parametrizations $\CL$ and $\CL^{ss}$ are compatible with homomorphisms
$$\Upsilon:  G_1 \ra G_2$$
of reductive groups where  the image of $\Upsilon$ is a normal subgroup, in the following sense.
Let ${}^L\Upsilon:  {^L}G_2 \ra {^L}G_1$ be the homomorphism of L-groups corresponding to $\Upsilon$.  Then for any cuspidal automorphic (resp. admissible irreducible) representation $\pi$ of $G_2$ over $K$ (resp. over $F$), 
$$\CL(\pi \circ \Upsilon) = {}^L\Upsilon\circ \CL(\pi)$$
(resp.
 $$\CL^{ss}(\pi \circ \Upsilon) = {}^L\Upsilon\circ \CL^{ss}(\pi).$$   
 This applies in particular if $\Upsilon$ is an isogeny of  semisimple groups.
\end{thm}
\begin{proof}  In the situation of (v), the condition $\pi^{I(p)} \neq 0$ implies that $\pi$ contains an eigenvector for $I$, and thus
%{\color{red} [Reference?]} 
is a constituent of a principal series representation. The assertion (v) then follows from (iv).  Assertion (vi) in the local setting is the last assertion of \cite[Th\'eor\`eme 0.1]{GLa}; in the global setting it is \cite[Proposition 12.5]{Laf18}.
\end{proof}

\begin{cor}\label{locsys}  Let $\Pi \in \CA_0(G)$, and suppose $|X| \subset Y$ is the largest closed subset such that $\Pi_v$ is unramified for $v \notin |X|$.   

(i) The parametrization of Theorem \ref{paramVL} defines a symmetric monoidal functor from
the category $Rep({}^LG)$ of finite-dimensional algebraic representations of ${}^LG$ to the category of completely reducible local systems on $Y \setminus |X|$:  to any $\tau:  {}^LG \ra GL(N)$ one attaches the $N$-dimensional local system attached to the
representation
$$\CV(\Pi,\tau) := \tau \circ \CL(\Pi)$$
of $\pi_1(Y \setminus |X|)$ (any base point).
\vskip 5pt

(ii)  Similarly, let $\sigma \in \CA(G,F)$.   The local parametrization $\CL^{ss}$ of \eqref{localparam} defines a symmetric monoidal functor from
the category $Rep({}^LG)$ of finite-dimensional algebraic representations of ${}^LG$ to the category of completely reducible
representations of $W_F$:
$$\tau \mapsto \CL^{ss}(\sigma,\tau) := \tau\circ \CL^{ss}(\sigma).$$
\vskip 5pt

(iii) Suppose $\Pi$ has the property that $\CV(\Pi,\tau)$ is abelian and semisimple for some faithful representation $\tau$.  Then the image of $\CL(\Pi)$, intersected with $\hG$, lies in a maximal torus of $\hG$. 
\end{cor}
\begin{proof}  The first assertion is the standard  characterization of $\hG$-local systems as symmetric monoidal functors; the second assertion is obvious.
\end{proof}
\vskip 10pt

We conclude this section with the following definition.
\vskip 5pt

\begin{defn}  With notation as in Theorem \ref{parab}, let $\sigma \in \CA(G,F)$.  We say $\sigma$ is {\bf pure} if, for some
(equivalently, for all) faithful representations $\tau:  {}^LG \ra GL(N)$, $\CL^{ss}(\sigma,\tau)$ has the property that, for 
any Frobenius element $Frob_F \in W_F$, the eigenvalues of $\CL^{ss}(\sigma,\tau)(Frob)$ are all Weil $q$-numbers of the same 
weight.
\end{defn}

\vskip 10pt

\section{Pure Weil-Deligne parameters}\label{pureWD}

This section is devoted to the proof of Theorem \ref{T:Lparameter}. We shall begin with some background material on tempered L--parameters or their equivalent incarnation as pure Weil-Deligne parameters.

\vskip 5pt

\subsection{Tempered parameters and pure Weil-Deligne parameters}

Let $F$ be a non-archimedean local field. 
\vskip 5pt

\begin{defn}
 A Weil-Deligne parameter for $F$ with values in ${}^LG(C)$ is a pair $(\rho,N)$ where 
$\rho:  W_F \ra {}^LG(C)$ is a continuous homomorphism (for the discrete topology on $C$) and $N \in Lie(G)(C)$ is a nilpotent element such that
$(Ad\circ\rho,N)$ is a homomorphism from the Weil-Deligne group of $F$ to $Aut(Lie(G))(C)$.  
\end{defn}

Without loss of generality we may assume $C = \CC$.  There is a simple recipe for converting a Weil-Deligne parameter $(\rho,N)$ for $F$ with values
in ${}^LG(\CC)$ to an L-parameter
$$\varphi = \varphi_{\rho,N}:  W_F \times SL(2,\CC) \ra {}^LG(\CC),$$
which is algebraic (over $\RR$) in the second factor, and vice versa.  This recipe  (and the proof that it gives an equivalence between the two notions)  is given in \cite[\S 2.1 and Prop. 2.2]{GR}.

%Suppose $G = GL(d,F)$, so ${}^LG(\CC) = GL(d,\CC)\times W_F$ and 
%$\rho \ra W_F \ra GL(V)$ is a representation of $W_F$ on a $d$-dimensional complex vector space $V$.   The map from $W_F \times SL(2,\CC)$ to the second
%factor $W_F$ of ${}^LG(\CC)$ is just the projection to $W_F$ and we ignore it in what follows.  Write $V = \oplus_i V_i$ as a sum of
%Jordan blocks for $N$, with $N_i = N~|_{V_i}$, so that each $N_i$ is of order $d_i = \dim V_i - 1$.   The restriction of $\varphi$ to $\{1\} \times SL(2,\CC)$
%has image in $\prod_i GL(V_i)$ and induces the irreducible $d_i$-dimensional representation $\psi_i$ of $SL(2,\CC)$ on $V_i$, in such a way that
%$u_i = exp(N_i)$ is identified with $\psi_i(\begin{pmatrix} 1 & 1 \\ 0 & 1\end{pmatrix})$ on $V_i$.   Since $\rho(I_F) \subset \rho(W_F)$
%commutes with $N$ it stabilizes each $V_i$ and $\varphi~|_{I_F} = \rho~|_{I_F}$.  Finally, choose a Frobenius element $Frob_F \subset W_F$
%so that $\rho(Frob_F) = s$ satisfies $su_is^{-1} =  u_i^q$, and define 
%$$\varphi(Frob_F \times 1) : = \rho(Frob_F)\times \psi_i(\begin{pmatrix} q^{\frac{1}{2}} & 0 \\ 0 & q^{-\frac{1}{2}}\end{pmatrix}).$$
%The equivalence class of $\varphi$ is then well-defined.

%For general $G$, we choose a faithful representation $\sigma:  {}^LG(\CC) \hookrightarrow GL(d,\CC)$ for some $d$.  Then
%$\varphi_{\sigma\circ\rho,\sigma\circ N}$ factors (uniquely) through a homomorphism $\varphi_{\rho,N}:  W_F \times SL(2,\CC) \ra  {}^LG(\CC)$.

We define a couple of properties of Weil-Deligne parameters.
\vskip 5pt

\begin{defn}  The Weil-Deligne parameter $(\rho,N)$ is {\rm tempered} if the restriction of $\varphi_{\rho,N}$ to $W_F \times SU(2)$ has bounded image after projection to $\hat{G}(\CC)$.
The parameter $(\rho,N)$ is {\rm essentially tempered} if its image in $\hat{G}^{ad}$ is bounded.  
\end{defn}

\vskip 5pt

\begin{defn}\label{WDpure}  (a) The Weil-Deligne representation $(\rho,N)$ with values in $GL(V)$ is pure if there is a complex number $t$ such that
\begin{itemize}
\item[(i)]  The eigenvalues of $\rho_t(Frob_F) := \rho(Frob_F)q^t$  are all $q$-numbers of integer weight.
\item[(ii)]  The subspace $W_aV \subset V$ of eigenvectors for $\rho_t(Frob_F)$ with eigenvalues of weight $\leq a$ is invariant
under $(\rho,N)$;
\item[(iii)]  Letting $gr_aV = W_aV/W_{a-1}V$, there is an integer $w$ such that, for all $i \geq 0$, the map
$$N:  gr_{w-i}V \ra gr_{w+i}V$$
is an isomorphism.
\vskip 5pt

\noindent The integer $w$ in (iii) is then called the weight of $(\rho,N)$ (twisted by $t$).
\end{itemize}
\vskip 5pt

(b)  Let $G$ be a connected reductive group over $F$.  The Weil-Deligne parameter $(\rho,N)$ with values in ${}^LG(\CC)$
is pure if $(\sigma\circ\rho,N)$ is pure of some weight for some (equivalently every) faithful representation $\sigma$ of ${}^LG$.
\end{defn}

In particular, if $N = 0$, then the Weil-Deligne parameter $(\rho,N)$ is pure if and only if the Weil group parameter $\rho$ is pure (in the usual sense) of some weight
(up to twist by a power of the norm).   We distinguish the two notions of purity by referring to ``pure Weil-Deligne parameters" and ``pure Weil parameters", or ``pure semi-simple
parameters," respectively.  In \cite{TY}, pure Weil pure parameters are called  {\it strictly pure}.

\begin{remark}\label{ipure}  The local parameters $\CL^{ss}(\pi)$ attached by Genestier and Lafforgue are $\ell$-adic.  In order to define what it means for them to be pure
(as Weil or Weil-Deligne parameters), one usually chooses an isomorphism $\iota:  \Qlb \isoarrow \CC$ (or more sensibly, an isomorphism between the algebraic
closures of $\QQ$ in the two fields) and declares them to be $\iota$-pure if they are pure in the above
sense after composing with $\iota$.   For this purpose one would want to replace the complex number $t$ in Definition \ref{WDpure} by an $\ell$-adic number.  In practice,
the parameters that arise in the automorphic theory are $\iota$-pure for every $\iota$.
\end{remark}

% Purity in previous sections referred exclusively to the latter notion.  We have tried to determine whether there is a convention in the literature to designate the Weil-Deligne parameters $(\rho,N)$ that satisfy the  conditions of Definition \ref{WDpure}
%when $\rho$ is not itself pure.  
%The article \cite{TY} calls such parameters pure but then uses the expression {\it strictly pure}, which does not seem to have caught on, for
%what we are calling pure Weil group parameters.    
\vskip 5pt

The following facts are well-known and in any case are easy to verify.
\vskip 5pt

\begin{lemma}   \label{L:purecompletion}
 (a)  The Weil-Deligne parameter $(\rho,N)$ is pure if and only if $\varphi_{\rho, N}$  is essentially tempered.
\vskip 5pt

(b)  Let $\rho:  W_F \ra {}^LG(C)$ be a semisimple parameter.  There is at most one way to complete $\rho$ to a  pure Weil-Deligne parameter $(\rho,N)$.  If it exists, we call it the pure completion of $\rho$. Equivalently, there is at most one essentially tame L-parameter $\varphi$ whose associated semisimple parameter (via (\ref{E:ssL})) is $\rho$.

\end{lemma}

\subsection{An application of the purity of the monodromy weight filtration}

The parametrization of Genestier-Lafforgue only attaches a semisimple parameter $\CL^{ss}(\pi)$ to a local representation $\pi$.  But when $\pi$ is realized as a local component of a cuspidal automorphic representation $\Pi$ of $G$ over some global function field $K$ -- say $\pi = \Pi_x$ for some place $x$ of $K$ --
then the restriction $\CL(\Pi)_x$ to the decomposition group $\Gamma_x$ at $x$ need not be semisimple.   We let 
\[  \CL_\Pi(\pi) = (\CL^{ss}(\pi),N_\Pi) \]
denote the Weil-Deligne parameter associated to the Galois parameter $\CL(\Pi)_x$.

The following result of Sawin and Templier generalizes Deligne's theorem on the purity of the monodromy weight filtration to $G$-local systems:

\begin{lemma}\label{mwf}\cite[Lemma 11.4]{ST} Let  $X$ be a smooth projective curve over a finite field $k$, with function field $K = k(X)$, and let 
$\rho:  Gal(K^{sep}/K) \ra {}^LG(C)$ be an irreducible homomorphism (with image not contained in a proper parabolic subgroup).  
Let $z$ be any place of $K$, and let $(\rho^{ss}_z,N_z)$ be the Weil-Deligne parameter associated to the restriction $\rho_z$ of $\rho$ to a decomposition group at $z$.  Then $(\rho^{ss}_z,N_z)$ is a pure Weil-Deligne parameter.
\end{lemma}

\begin{remark}  The lemma is stated in \cite{ST} when $G$ is split semisimple but the proof is valid for any reductive $G$, with the appropriate extension of the property
of being contained in a proper parabolic subgroup to disconnected ${}^LG(C)$.  
\end{remark}

%{\color{red}  Question for Will:  Is this true when $G$ is not split?}

\begin{corollary}\label{pureN}  Let $\pi$ be an irreducible admissible representation of $G(F)$.  Suppose 
\vskip 5pt

\begin{itemize}
\item $X$ is a smooth projective curve over a finite field $k$
\item  $z$ is a place of the function field $K = k(X)$ such that $K_z \isoarrow F$.  
 %an $L$-homomorphism  $\sigma:  {}^LG \ra {}^LGL(d)$ whose kernel is contained in the center of $\hat{G}$ and 
\item there is a cuspidal automorphic representation $\Pi$ of $G_K$,%, unramified outside the finite set $S$,
with $\Pi_z \isoarrow \pi$, such that the global parameter $\CL(\Pi)$ is irreducible.
\end{itemize}
Then the Weil-Deligne parameter $\CL_\Pi(\pi)$ is pure.  Moreover, it is the  unique pure completion of the
semisimple parameter $\CL^{ss}(\pi)$.
\end{corollary}
Not every semisimple parameter admits a completion to a pure Weil-Deligne parameter.  Corollary \ref{pureN} asserts that $\CL^{ss}(\pi)$ does admit a pure completion provided
$\pi$ can be realized as a local component of a cuspidal  automorphic representation whose global parameter is irreducible in the indicated sense.
\vskip 10pt

\subsection{Proof of Theorem \ref{T:Lparameter}}

We are now ready to prove Theorem \ref{T:Lparameter}.  In view of Corollary \ref{pureN},  it suffices to construct cuspidal globalizations of any essentially  discrete series representation $\pi$ whose global parameter is irreducible. The following proposition achieves this when $\pi$ is supercuspidal.

    \vskip 5pt

\begin{prop}\label{mixed}  Let $G$ be a reductive group over $F$ and assume that
\begin{itemize}
\item $p > 2$  if $G$ is of type $B_n$ or $C_n$; 
\item $p > 3$  if $G$ is of type $G_2$ or $F_4$; 
\end{itemize}
(in other words, we assume $p$ is a good prime for these groups).   Let $\pi$ be a supercuspidal representation of $G(F)$.  Let $X$ be a smooth projective curve over a finite field $k$
and let $z$ be a place of the function field $K = k(X)$ such that $K_z \isoarrow F$.  Then there exist infinitely many cuspidal automorphic representations
$\Pi$ of $G_K$, such that
\begin{itemize}
\item  $\Pi_z \isoarrow \pi$;
\item  the global parameter $\CL(\Pi)$ is irreducible.  
\end{itemize}
In particular, the semisimple parameter
$\CL^{ss}(\pi)$ admits a completion to a pure Weil-Deligne parameter that is realized as the specialization at $z$ of the global parameter $\CL(\Pi)$, for any such $\Pi$.
\end{prop}

\begin{proof}  
 We shall use the results of \cite{GLo} to produce the desired globalizations.
Let $x \neq z$ be any other place of $K$.  Since there are infinitely many such $x$, it suffices to show 
\begin{itemize}
\item[(i)] There is a supercuspidal representation $\pi'$ of $G(K_x)$ such that $\CL^{ss}(\pi')$ is irreducible, and 
\item[(ii)] There is an automorphic representation $\Pi$ of $G_K$ with $\Pi_x \isoarrow \pi'$ and $\Pi_z \isoarrow \pi$, and with $\Pi_y$ tamely 
ramified for all $y \neq x, z$.
\end{itemize}
Once (i) is given, the existence of $\Pi$ as in (ii) is proved in \cite{GLo}  by a Poincar\'e series construction.  It thus remains to produce a supercuspidal representation $\pi'$ as in (i). 
\vskip 5pt

We let $k' = k(x)$ denote the residue field at $x$ and construct the Kloosterman sheaf ${\rm Kl}_{\hat{G}}(\phi,\chi)$ of \cite{HNY} over $\mathbb{P}^1_{k'}$.  Let $\pi'$ be the local component at $\infty$ of the corresponding automorphic
representation $\pi(\phi,\chi)$.  Under the hypotheses on $G$ and $p$, \cite[Theorem 2]{HNY} implies that the restriction to the inertia group at $\infty$
of the local monodromy representation of ${\rm Kl}_{\hat{G}}(\phi,\chi)$ is irreducible.  On the other hand, it follows from construction
that ${\rm Kl}_{\hat{G}}(\phi,\chi)$ is the global parameter attached by \cite{Laf18} to $\pi(\phi,\chi)$.   More precisely, \cite{HNY} computes the local monodromy of ${\rm Kl}_{\hat{G}}(\phi,\chi)$  at unramified places and identifies it with that obtained from the Satake parameters of
$\pi(\phi,\chi)$.  (The construction of $\pi(\phi,\chi)$ is made explicit
 in \cite{Y16} but its existence is implicit in \cite{HNY}.)  It then follows that $\CL^{ss}(\pi')$ is irreducible, as desired.
\end{proof}
\vskip 10pt

For  essentially discrete series but not supercuspidal representations, we have a weaker globalization result which is nonetheless sufficient for our applications.
\vskip 5pt

\begin{prop}\label{mixedds}
Assume the notation and hypotheses of Proposition \ref{mixed} but let $\pi$ be an essentially  discrete series representation. Then there exists a cuspidal automorphic representation $\Pi$ of $G_K$ such that
\begin{itemize}
\item $\CL^{ss}(\Pi_z) = \CL^{ss}(\pi)$;
\item  the global parameter $\CL(\Pi)$ is irreducible.  
\end{itemize}
In particular, the semisimple parameter
$\CL^{ss}(\pi)$ admits a completion to a pure Weil-Deligne parameter $\CL(\pi)$ that is realized as the specialization at $z$ of the global parameter $\CL(\Pi)$, for any such $\Pi$.
\end{prop}

\begin{proof}
  Let $f_{\pi}$ be a pseudo-coefficient of $\pi$, whose properties are recalled in \cite[\S 8.1]{GLo}. 
  In \cite[\S 8.2]{GLo}, a version of the simple trace formula was formulated as a working hypothesis.   Applying Lemma \ref{L:appendix} in the appendix by R. Beuzart-Plessis,  it follows as 
 in the proof of  \cite[Proposition 8.2]{GLo} that there is  a cuspidal $\Pi$ with $\Pi_x \isoarrow \pi'$ (with $\pi'$ as in Proposition \ref{mixed}) for some place $x$ and $Tr(\Pi_z)(f_\pi) \neq 0$.
 This implies (by \cite[Lemma 8.1]{GLo}) that $\Pi_z$ has the same cuspidal support as $\pi$, and therefore that
$$\CL^{ss}(\Pi_z) = \CL^{ss}(\pi)$$
since the semisimple Genestier-Lafforgue parameter is compatible with parabolic induction.   Now we do not know, nor care, whether $\Pi_z$ is isomorphic to the discrete series $\pi$ or not.  Since   $\CL^{ss}(\Pi_x) = \CL^{ss}(\pi')$ is irreducible,   so is the global parameter $\CL(\Pi)$ and we have produced the desired $\Pi$. 
 \end{proof}
 \vskip 5pt
 
 Using Corollary \ref{pureN}, Proposition \ref{mixed} and Proposition \ref{mixedds}, the proof of Theorem \ref{T:Lparameter} is now complete.
\vskip 5pt
 
We note that the idea behind Corollary \ref{pureN} and Proposition \ref{mixed} was already applied in the case of classical groups  in the proof of \cite[Proposition 7.3]{GLo}.  Instead of
using the Kloosterman representations of \cite{HNY} to establish irreducibility, \cite{GLo} used depth zero supercuspidal representations, but introduced a hypothesis because the irreducibility of the corresponding L-parameter has not (yet) been established for local fields of positive characteristic.

\begin{cor}\label{weight0}  Under the hypotheses of Theorem \ref{T:Lparameter}, suppose $G$ is semisimple.  Then the composition of
the adjoint representation of ${}^LG$ with the Weil-Deligne parameter $\CL(\pi)$ of $\pi$ is pure of weight $0$.
\end{cor}

\begin{proof}  It follows from Proposition \ref{mixedds} that $Ad\circ \CL(\pi)$ is a pure Weil-Deligne representation of some weight $w$.  But since $G$, and therefore $\hat{G}$, is semisimple, the determinant of $Ad\circ \CL(\pi)$ is pure of weight $0$.  It follows that $w$ must equal $0$.
\end{proof}

\begin{remark}\label{ipure2}  The proof actually shows that $\CL(\pi)$ is $\iota$-pure, in the sense of Remark \ref{ipure}, for every $\iota$.
\end{remark}

\vskip 10pt

\section{Local systems on open subsets of $\PP^1$}  \label{S:local}

The rest of the paper is devoted to establishing more refined properties of the semisimple parameter $\CL^{ss}(\pi)$ of a supercuspidal representation $\pi$ and in particular proving Theorem \ref{mainthm}. For this, we shall need to appeal to a global argument which necessitates the construction of appropriate globalizations of supercuspidal representations by Poincar\'e series.
Unlike the previous section, these globalizations will have to be constructed over the function field of $\mathbb{P}^1$, so that we may appeal to special properties of local systems on open subsets of $\mathbb{P}^1$. In this section, we shall discuss the requisite results about local systems on open subsets of $\PP^1$.

\subsection{Purity}\label{purity}

Let $k = \Fq$ be a finite field of characteristic $p$. Let $Z, X \subset \PP^1(k)$ be disjoint finite subsets, $|Z|$ and $|X|$ the corresponding reduced subschemes of $\PP^1(k)$, and let $Y = \PP^1_k \setminus |X| \cup |Z|$, viewed as a scheme over $Spec(k)$. Let $L$ be a finite extension of $\Ql$ and consider a local system $\CV$ over $Y$ with coefficients in $L$ (a lisse $\ell$-adic sheaf). Picking a point $y$ of $Y$ (the generic point, for example),  and letting $V$ denote the finite-dimensional vector space $\CV_y$ over $L$, we can identify $\CV$ with a representation
$$\rho:  \pi_1(Y,y) \ra Aut(V).$$
%We will always assume $\det(\rho)$ to be a character of finite order.
Let $F_z = k(Y)_z$ be the completion of $k(Y)$ at the point $z \in Z$.
The monodromy of $\CV$ at $z$ is the restriction of $\rho$ (up to conjugation) to a homomorphism 
\begin{equation}\label{localmono} \rho_z:  Gal(F_z^{sep}/F_z) \ra Aut(V).
\end{equation}
Let $I_z \subset Gal(F_z^{sep}/F_z)$ denote the inertia group and let $Frob_z \in Gal(F_z^{sep}/F_z)$ denote any geometric Frobenius element; i.e.,
any lift to an automorphism of $F_z^{sep}$ of the automorphism $x \mapsto x^{\frac{1}{q}}$ of $\bar{k}$.  

If $\Gamma$ is a topological group and $\rho:  \Gamma \ra Aut(V)$ is a continuous
homomorphism, we let $\rho^{ss}$ denote its semisimplification -- the sum of its Jordan-H\"older constituents.

\begin{thm}[\cite{De80}, Corollary 1.7.6]\label{pure}  Suppose that $Z$ is not empty and, for all $z \in Z$, the following conditions are satisfied:
\begin{itemize}
\item $\rho_z^{ss}$ is unramified;
\item the Frobenius eigenvalues $\rho_z(Frob_z)$ are all Weil $q$-numbers of the same weight.
\end{itemize}
Then $\CV$ extends to a local system over $\PP^1_k \setminus |X|$.
\end{thm}
\begin{proof}   In fact, Deligne's result only requires the Frobenius eigenvalues to have the same weight for a single complex embedding; under this condition, $\rho_z(I_z)$ is a finite group.  But we have assumed that $\rho_z^{ss}$ is unramified, thus $\rho_z(I_z)$ is unipotent.  It
follows that $\rho_z(I_z)$ must be trivial.
\end{proof}

\begin{corollary}\label{abelian}  Under the hypotheses of Theorem \ref{pure}, suppose $X = \{0,\infty\}$ and $\CV$ is a semisimple local system on $\PP^1$ with associated monodromy representation 
$$\rho:  \pi_1(\PP^1 \setminus |X|,y) \ra Aut(V).$$
 Suppose moreover that
 \begin{itemize}
 \item[(a)] the local monodromy representation of $\CV$ at $0$ and $\infty$ is tame, 
 \item[(b)]  the semisimplification of the local monodromy representation at $0$ {\bf or} $\infty$ is abelian and Frobenius-semisimple. 
 \end{itemize}
  Then, possibly after replacing $L$ by a finite extension $L'$, $\CV$ breaks up as the sum of $1$-dimensional local systems.
\end{corollary}

\begin{proof}   Theorem \ref{pure} implies that $\CV$ extends to a tame local system, also called $\CV$, over $\Gm$.  By \cite{Laf02}, Corollary VII.8, $\CV$ can be written as a (finite) direct sum
$$\CV = \bigoplus_{\chi} \CV_\chi \otimes \chi$$
where each $\CV_\chi$ is a mixed local system and $\chi$ is the pullback to $\PP^1\setminus |X|$ of a rank $1$ $\ell$-adic local system over $Spec(k)$, in other words a continuous $\ell$-adic character of $Gal(\bar{k}/k)$.   We thus reduce to the case where $\CV$ is a mixed $\ell$-adic local system, and since $\CV$ is assumed semisimple, we may assume $\CV$ to be irreducible. 
\vskip 5pt

  Then the purity hypothesis at the points $z$ in the non-empty set $Z$ implies that $\CV$ is itself pure.  Since the tame fundamental group of $\GG_{m,\bar{k}}$ is just $\prod_{p' \neq p} \ZZ_{p'}(1)$, it follows from the purity that the image $J$ under $\rho$ of the geometric fundamental group of $\Gm$ is finite and necessarily abelian. 
  \vskip 5pt
  
  Now one knows that the image of the tame inertia group $I_0$ at $0$ in the tame fundamental group of $\GG_{m,k}$ is equal to the
  tame fundamental group of $\GG_{m, \bar{k}}$ (i.e. the tame geometric fundamental group of $\GG_{m,k}$). The same is true for the tame inertia group at $\infty$. 
  Hence, the hypothesis (b) on the local monodromy at $0$ or $\infty$ implies that the conjugation action of $Gal(\bar{k}/k)$ on $J$ is trivial, and thus the image
under $\rho$ of the arithmetic fundamental group $\pi_1(\PP^1 \setminus |X|,y)$ is also abelian; moreover, the image of $Gal(\bar{k}/k)$ is semisimple by hypothesis.  This completes the proof.
\end{proof}

\subsection{Kummer theory}

Let $K = k(\PP^1) = k(T)$ be the global function field of $\PP^1$, and let $K'$ denote the union of all abelian extensions of $K$ of degree dividing $q - 1$, which is the order of the group of roots of unity in $K$.  Homomorphisms 
$$\alpha:  Gal(K'/K) \ra \Qlb^\times$$ 
are classified by Kummer theory:
\begin{equation}\label{kummer}
\psi:  K^\times/(K^\times)^{q-1}  \isoarrow  X(K) := Hom(Gal(K'/K), \Qlb^\times).
\end{equation}
Let $C_T \subset K^\times/(K^\times)^{q-1}$ be the subgroup generated by $k^\times$ and the parameter $T$, and let $X_T = \psi(C_T)$.
We also consider Kummer theory for the constant field $k$:
\begin{equation}\label{kummerk}
\psi_k:  k^\times  \isoarrow  X(k) := Hom(Gal(k'/k), \Qlb^\times),
\end{equation}
where $k'$ is the cyclic extension of $k$ of order $q^{q-1}$. For $b \in k^\times$, $\beta \in Gal(k'/k)$, let us write:
\[   \beta^b := \psi_k(b)(\beta). \]

Any $z \in k^\times$ defines a geometric point $z \in \GG_{m,k}$; let $\Gamma_z \subset Gal(K'/K)$ denote its decomposition group and  $I_z$ its inertia group.   An element $\alpha \in X_T$ is unramified at such a point $z$, and thus defines a homomorphism 
$$\alpha^{(z)}: \Gamma_z/I_z  \isoarrow Gal(k'/k)  \ra \Qlb^\times.$$
For $a = 1, \dots, q-1$ we let $\alpha_a = \psi(T^a)$.   Then we have the following reciprocity law:

\begin{lemma}\label{recip}  For $z \in \GG_{m}(k)$ and $a = 1, \dots, q-1$ we have
$$(\alpha_a)^{(z)} = \beta^{az}.$$
\end{lemma}
\begin{proof}  Let $K_z$ denote the completion of $K$ at $z$, $\CO_z$ its integer ring, $k_z$ its residue field, which the inclusion of $k$ in $K$ canonically identifies with $k$.    The image of the element $T \in \CO_z$ in $k_z$ is identified with the element $z \in k^\times$,  so the residue field of the unramified extension $\CO_z[T^{1/q-1}]$ is just $k(z^{1/q-1})$.   The identity is then obvious.
\end{proof}

\begin{lemma}\label{Frobenius-trace} Let $\mathcal L$ be a rank one local system on $\mathbb G_{m,k}$ and $z \in k^{\times}$. Then
\[ \operatorname{tr}( \operatorname{Frob}_z, \mathcal L) =  \operatorname{tr}(\operatorname{Frob}_1 , \mathcal L)  \operatorname{tr}(m_z ^{-1} , \mathcal L) \]
where $m_z$ is the element of the inertia group $I_0$ which acts by multiplication by $z$ on the $(q-1)$st root of $T$. 
\end{lemma}

\begin{proof} 
We give two proofs of the lemma.

\vskip 5pt

\noindent (i) (Class field theory) It suffices to show that  $\operatorname{Frob}_z = \operatorname{Frob}_1 \cdot m_z^{-1} $ in the abelianization of the tame fundamental group of $\mathbb G_m$. Class field theory gives an isomorphism between this abelianization and the profinite completion of \[ K^\times \backslash \mathbb A_K^\times /  \left( \prod_{\substack{ y \in |\mathbb P^1_k|\\ y \notin \{0,\infty|}} \mathcal O_{K_y}^\times \times \prod_{y \in \{0,\infty\}} U_{K_y, 1} \right). \] 
Moerover, 
\begin{itemize}
\item $\operatorname{Frob}_z$ corresponds to the id\`ele which is a uniformizer at $z$ and the identity at all other places, 
\item $\operatorname{Frob}_1$ corresponds to the id\`ele which is a uniformizer at $1$ and the identity at all other places, and
\item  $m_z$ corresponds to the idele which is $z$ at $0$ and the identity at all other places. 
\end{itemize}
So it suffices to check that the first id\`ele class is the product of the other two id\`ele classes. This can be checked by multiplying by $\frac{T- z}{T-1} \in K^\times$, which is a uniformizer at $z$, the inverse of a uniformizer at $1$, and $z$ at $0$.
\vskip 10pt

\noindent (ii) (Kummer theory) We can construct a one-dimensional representation of $Gal(\overline{k}/k)$ on which $\operatorname{Frob}_q$ acts by multiplication by $ \operatorname{tr}(\operatorname{Frob}_1 , \mathcal L)$. On tensoring $\mathcal L$ with the inverse representation, we reduce to the case where $\operatorname{Frob}_1$ acts trivially on $\mathcal L$. Since the tame fundamental group of $\mathbb G_m$ is generated by the geometric tame fundamental group and Frobenius, the image of the Galois group acting on $\mathcal L$ must equal the image of the geometric tame fundamental group. The action of the geometric tame fundamental group on $\mathcal L$ must be by a $\operatorname{Frob}_q$-invariant character, which means it factors through the $\operatorname{Frob}_q$-coinvariant of the abelianization of the tame geometric fundamental group, which is $\mathbb Z/(q-1)$.  So the action of the geometric fundamental group is by a character of order $q-1$.  The result now follows from Lemma \ref{recip}.

\end{proof}

 \vskip 10pt

 \section{Unramified Reductive Groups}\label{S:unramred}
 We continue to assume that $K = k(\PP^1)$, with $k = \mathbb{F}_q$. In this section, we establish some structural results for an unramified connected reductive group $G$ over $K$.
 In particular, we establish a useful result (Lemma \ref{parahoric-coordinates-lemma}) which gives congruence conditions on the matrix entries of a parahoric subgroup  of $G$ (at a place of $K$) under a faithful algebraic representation.  This parahoric entry inequality will be crucially used in the next few sections.
 \vskip 5pt
 
\subsection{Unramified groups}\label{SS:unramified}
By definition, a connected reductive group over $K$ is unramified if it is quasi-split and split by an unramified extension of $K$. Such a $G$ can in fact be defined over the finite field $k$. Hence,  we assume henceforth that $G$ is defined over $k$, with $Z$ the identity component of its center and    $B = T \cdot N$ a fixed  Borel $k$-subgroup. 
Let $S \subset T$ be the maximal $k$-split subtorus in $T$, so that $S$ is a maximal split torus of $G$ over $k$. 
\vskip 5pt

Let 
\[  \rho\colon G \to \operatorname{Aut}(V) \]
be a faithful algebraic representation of $G$ defined over $k$, and let $e_1,\dots, e_n$ by a basis of $V$  on which $S$ acts diagonally via  the eigencharacter $\lambda_i \in X^* (S)$ on $e_i$.

\vskip 5pt

By base change, we obtain from the above the analogous objects over any $k$-algebra. In particular, we have:
\begin{itemize}
\item  $S_K \subset B_K = T_K \cdot N_K \subset G_K$ over $K = k(\PP^1) = k(t)$; to simplify notation, we shall suppress the subscript $K$.  Note that $S$ is still a maximal $K$-split 
torus in $G$.
\vskip 5pt

 \item At each place $y$ of $K$, with associated local completion $K_y$ and ring of integers $\mathcal{O}_y \subset K_y$, we have the $\mathcal{O}_y$-group schemes
\[ S_y \subset B_y  = T_y\cdot N_y \subset G_y. \]
 Note however that $S_y$ need no longer be a maximal $K_y$-spit torus in $G_y$; we let $A_y \supset S_y$ be a maximal split torus over $\mathcal{O}_y$.  For $y \in \PP^1(k)$, we have $A_y = S_y$. 
\vskip 5pt

\item At each place $y$ of $K$, $G(\mathcal{O}_y)$ is a hyperspecial maximal compact subgroup of $G(K_y)$. One has the reduction-mod-p  map
\[  G(\mathcal{O}_y) \longrightarrow G(k_y)  \]
where $k_y$ is the residue field of $K_y$. The preimage of $B(k_y) \subset G(k_y)$ in $G(\mathcal{O}_y)$ is an Iwahori subgroup $I_y$ of $G(K_y)$. Its pro-p radical $I(p)$ is the preimage of the unipotent radical $N(k_y)$ of $B(k_y)$.

\vskip 5pt
\item For each place $y$ of $K$, we have a faithful algebraic representation $\rho_y \colon G_y \to \operatorname{Aut}(V_y)$. If $y \in \PP^1(k)$, then the basis $\{e_i \}$ is a basis of eigenvectors for the maximal $K_y$-split torus $S_y = A_y$.

\end{itemize}

\vskip 5pt

\subsection{Buildings}  \label{SS:buildings}
For each place $y$ of $K$, consider the (extended) Bruhat-Tits building $\mathcal{B}(G_y)$  which contains the apartment $\mathcal{A}(A_y)$ associated to the maximal $K_y$-split torus $A_y$. The hyperspecial maximal compact subgroup $G(\mathcal{O}_y)$ determines a basepoint in $\mathcal{A}(A_y)$, giving an identification
\[  \mathcal{A}(A_y)  \isoarrow X_*(A_y) \otimes \RR, \]
with $G(\mathcal{O}_y)$ the stabilizer of the origin $0 \in X_*(A_y)$. 
For places $y \notin \PP^1(k)$, we shall only need to consider this basepoint  $0$ and its stabilizer $G(\mathcal{O}_y)$. 
\vskip 5pt

Now consider only those $y \in \PP^1(k)$, so that $A_y  = S_y$. In this case, one has a natural identification of the apartment $\mathcal{A}(A_y)$ with $X_*(S) \otimes \RR$. In particular, the apartments $\mathcal{A}(A_y)$ are naturally identified with each other as $y$ varies in $\PP^1(k)$. 
 
\vskip 5pt

\subsection{ Parahoric entry inequality}  
In this section, we  prove a technical lemma on the matrix entries of elements of parahoric subgroups of $G$ over the local field $K_y= k((t))$ (for $y \in \PP^1(k)$)  that will be useful for us later on.
 
\vskip 5pt

\begin{lemma}\label{parahoric-coordinates-lemma} 
Fix a local field $F= k((t))$ (which we may think of as $K_y$ for $y \in \PP^1(k)$) and let
 
 \begin{itemize}
 \item $G$ be an unramified {\it semisimple} group (as introduced above), with maximal $F$-split torus $S$ contained a Borel subgroup $N = TN$, all defined over $k$;
\item  $c$ be a point in the apartment of $S$ in the Bruhat-Tits building $\mathcal{B}(G)$  with associated parahoric subgroup $G_c$;
\item  $\rho\colon G \to \operatorname{Aut}(V) $ be a faithful algebraic representation of $G$, with $S$-eigenbasis $e_1,\dots, e_n$  with associated eigencharacters $\lambda_i \in X^* (S)$ on $e_i$.
\end{itemize}
Then for any $g \in G_c \subset G(K)$, we have 

\begin{equation}\label{parahoric-entry-inequality} 
v ( \rho_{ij} (  g)) \geq c \cdot (\lambda_i - \lambda_j ),\end{equation} 
where $v: F \to \ZZ$ is the normalized valuation on $F$ and $\rho_{ij}(g)$ denotes the $(i,j)$-th entry of the matrix of $\rho(g)$ relative to the basis $\{e_i\}$.
\vskip 5pt

More generally,  if $g$ lies in the Moy-Prasad subgroup $G_{c,r+}$ for some $r \geq 0$, then we have
 \begin{equation}\label{parahoric-entry-inequality-2} 
 v ( \rho_{ij} (  g) - \delta_{ij} ) > c \cdot (\lambda_i - \lambda_j ) +r.\end{equation} 
\end{lemma} 

\begin{proof} We give the proof in the $G_c$ case, and then describe the modifications necessary for $G_{c,r+}$.
\vskip 5pt

If \eqref{parahoric-entry-inequality} holds for $g_1$ and $g_2$, then it holds for $g_1g_2$, so that it suffices to verify \eqref{parahoric-entry-inequality} for a set of generators of the parahoric $G_c$. The parahoric $G_c$ is generated by the subgroups $U_\psi$ for those affine functions $\psi$ on the apartment of $S$ such that 
\[  \psi ( c)  \geq 0. \]
  Here $\psi$ is an affine function of the form $\psi(w) = b \cdot w -s$ for a root $b \in X^*(S)$ and a real number $s$. If $U_b \subset G$ is the (restricted) root group associated to $b$, then one has  an isomorphism over $k$:
  \[ i_b:  U_b \isoarrow \mathbb {\rm Res}_{k_b/k}G_a, \] 
  and the subgroup $U_\psi$ can be described as:
  \[  U_{\psi} = \{  u \in U_b(F):  v(i_b(u)) \geq s \}, \]
  where $v$ is the unique extension of $v$ to $k_b$. 
\vskip 5pt

To obtain \eqref{parahoric-entry-inequality}, it suffices to check that the following two statements  for $g\in U_b$: 
\vskip 5pt

\begin{itemize}
\item[(a)]  If $\rho_{ij} (g) \neq 0$, then $\lambda_i -\lambda_j = n b$ for some $n\geq 0$;
\vskip 5pt

\item[(b)]  $\rho_{ij}(g)$ is a homogeneous polynomial function on $U_b$ defined over $k$ of degree $n$.
\end{itemize}
Indeed, if (a) and (b) hold, then we have:
 \[v (\rho_{ij}(g)) \geq n s   \geq ns + n \psi ( c) = n b \cdot c= (\lambda_i -\lambda_j) \cdot c,\]
 as desired.
\vskip 5pt

To prove (a) and (b), we use lifting to characteristic $0$.  
The group $G$ (which is defined over $k$) has a lift from the field $k$ to the Witt vectors $W(k)$, unique up to not-necessarily-unique isomorphism. Since irreducible algebraic representations of $G$  are classified by Galois orbits of highest weights, and the Galois action for $k$ and $W(k)$ are the same, we may lift the representation $V$ from the field $k$ to its ring  of Witt vectors $W(k)$.  Both properties in (a) and (b)  are preserved by reduction mod $p$, and thus may be checked over the fraction field of $W(\kappa)$, where $U_b$ is obtained by exponentiating the root  subspace of $b$. In characteristic $0$, both (a) and (b) follow readily from the representation theory of $SL_2$. 
\vskip 5pt

We now indicate the modifications needed for the statement for $G_{c, r^+}$.
The inequality  \eqref{parahoric-entry-inequality-2} is also stable under matrix multiplication since $r \geq 0$. Because the subgroups $G_{c,r+}$ are generated by $U_\psi$ for those $\psi$ with $\psi(c) > r$ and the subgroups $T^a_{n \delta}$ for the maximal torus $T$ containing $S$, for $n > r$.  Here, $a$ is an (absolute) simple root of $T$ (over $\overline{k}$), with associated root subgroup $SL_2$ (over $\overline{k}$),  $T^a$ is the 1-dimensional torus (over $\overline{k}$) which is the image of the coroot   
\[  a^{\vee}: \mathbb{G}_m \to T^a \subset T \]
and 
\[  T^a_{n \delta} = \{ a^{\vee}(x)  \in T^a(F^{ur}) \cap G(F):  v(x-1) \geq n \}. \]
\vskip 5pt

We begin with the elements of $U_{\psi}$ as above. 
That the elements of  such $U_\psi$ satisfy \eqref{parahoric-entry-inequality-2} follow from (a) and (b) above  and the additional fact that $\rho_{ij}$ is the constant function $1$ on $U_b$ if $n=0$.  Indeed, these imply that:
\[v (\rho_{ij}(g)) \geq n s   \geq ns + n \psi ( c)  +nr = n b \cdot c + nr> (\lambda_i -\lambda_j) \cdot c+r \]
if $n>0$ and $v (\rho_{ij}(g)-1) = \infty$ if $n=0$.
\vskip 5pt

Hence, it remains to show that elements  $g\in T^a_{n \delta}$ satisfy \eqref{parahoric-entry-inequality-2}. Because $T$ is the centralizer of  $S$ in $G$, we see that 
\[  \text{$ g \in T^a_{n \delta}$ and $\lambda_i \ne \lambda_j$} \Longrightarrow    \rho_{ij}(g)=0. \]  
Thus, it suffices to show 
\[  v(\rho_{ij}(g)-\delta_{ij}) \geq n > r \quad \text{ if $\lambda_i = \lambda_j$.} \]
Now  the action of the elements $g \in T^a$ and therefore the matrix entries $\rho_{ij}$ (restricted to $T^a$)  are given by polynomial functions on $\mathbb G_m$ over $\overline{k}$, which are equal to $\delta_{ij}$ at the identity element $1$. Thus the difference between $\rho_{ij}(g)$ and $1$ has valuation at least $n$ for $g \in T^a_{n \delta}$, as desired. 
 %Here $a$ is a simple root of $Z$ in $G$, and $Z^a_{n\delta}$ is defined as a subset of the maximal torus $Z_a$ of the copy of $SL_2$ arising from the root $a$ consisting of elements that, in the natural identification of the maximal torus of $SL_2$ with $\mathbb G_m$, differ from $1$ by an element of valuation at least $n$. These roots and this copy of $SL_2$ may be defined over $\overline{\kappa}$, so the action of the torus must be, and therefore the matrix entries are given by polynomial functions on $\mathbb G_m$ over $\overline{\kappa}$, which are equal to $\delta_{ij}$ at $1$, and thus the difference between $\rho_{ij}(g)$ and $1$ has valuation at least $n$, as desired. 

 \end{proof}

\vskip 5pt

\section{ Open Compact Subgroups of $G(\mathbb{A}_K)$}  \label{S:open}
After the preparatory local lemma above, we return to the global setting of \S \ref{SS:unramified} where  $G$ is an unramified semisimple group over $K = k(t)$.
We will thus use the notation of \S \ref{SS:unramified} and \S \ref{SS:buildings}.
Our goal is to construct some open compact subgroups of $G(\mathbb{A}_K)$ with desirable properties.
\vskip 5pt

With $F = k((t))$,  let $U^* \subset G(F)$ be a fixed maximal compact subgroup.
We choose an open compact subgroup of 
$G(\ad_{k(t)})$ of the form $U = \prod_x U_x$, where $x$ runs over the places of $K$, with the following properties.
\begin{itemize}
\item[(a)]  At every $z \in \Gm(k) \subset \PP^1(k)$, $U_z$ equals the fixed $U^*$, with respect to a chosen isomorphism 
$$G(K_z) \isoarrow G(F);$$
\item[(b)]  For $x \notin |\PP^1(k)|$, $U_x$ is the hyperspecial maximal compact subgroup $G(\CO_x)$;
\item[(c)]  The subgroups $U_0$ and $U_\infty$ are parahoric subgroups that remain to be chosen.
\end{itemize}

\begin{prop}\label{parah}  
Let $G$ be unramified and simply-connected. The subgroups $U_0$ and $U_\infty$ can be chosen in such a way that
$$G(K)\cap U = T(k).$$
\end{prop}
\vskip 5pt

\begin{proof}
Let $\Lambda = X_*(S)$ be the cocharacter lattice of $S$. We have remarked that for $z \in \PP^1(k)$,  we have a natural identification of 
$\Lambda\otimes\RR$ as an apartment in the Bruhat-Tits building of $G_z$. Let $c \in \Lambda\otimes\RR$ be the fixed point of $U^* \subset G(F) \cong G(K_z)$; because $G$ is simply-connected, $U^*$ is the parahoric subgroup $G_c$. 
For $y \notin \PP^1(k)$, we also have a natural inclusion $\Lambda \otimes \RR \hookrightarrow \mathcal{A}(A_y)$.
%Suppose pi is induced from the stabilizer of y (are these the subgroups that appear in Fintzen's work?).
Choose $a, b \in \Lambda\otimes\RR$ such that 
\[ a+b+(q-1)y =0 \] 
and such that $a$ is sufficiently generic.
\vskip 5pt

For a place $v$ of $K$, define $f(v)$   by:
\[    \begin{cases}
 f(0) = a,  \\
 f(\infty) = b \\
 f(z) = c \text{   for $z \in \Gm(k)$,} \\ 
 f(x) = 0,  \text{  for $x \notin |\PP^1(k)|$.} \end{cases} \]
 Under this definition, we take $U_v $ to be the parahoric subgroup $G_{f(v)}$ associated to $f(v)$.
\vskip 5pt

We are now ready to prove that $G(K) \cap U = T(k)$. To do this, we use the faithful algebraic representation $(\rho, V)$ fixed in \S \ref{SS:unramified}, equipped with a basis $e_1,\dots, e_n$, where  $e_i$ is an eigenvector of $S$ with eigencharacter $\lambda_i \in X^*(S)$. 
Applying Lemma \ref{parahoric-coordinates-lemma}, we see that for $g \in U \cap G( K)$ with $\rho_{ij}(g) \neq 0$ (for a fixed pair $(i,j)$),  we have  
\[ 0 = \sum_v v (\rho_{ij}(g)) \geq \sum_v f(v) \cdot (\lambda_i - \lambda_j) = 0 \cdot (\lambda_i - \lambda_j) =0 .\] 
Because both sides are $0$, the inequality is sharp, so  \[v (\rho_{ij}(g)) =  f(v) \cdot (\lambda_i - \lambda_j) \] for all $v$. This implies that $f(v) \cdot (\lambda_i - \lambda_j)$ is an integer, and in particular that $a \cdot (\lambda_i -\lambda_j)$ is an integer. However, since $a$ is generic by assumption, we can certainly assume that $ a \cdot (\lambda_i -\lambda_j)$ is not an integer for any pair $(i,j)$ for which  $\lambda_i \ne \lambda_j$.
\vskip 5pt

We conclude that, for $g \in U$, $\rho_{ij}(g) =0$ unless $\lambda_i = \lambda_j$. In other words, $\rho(g)$ commutes with $\rho(S)$, and thus $g \in T(K)$. 
Furthermore, for a pair $(i,j)$ with $\rho_{ij}(g) \neq 0$, we have 
\[v ( \rho_{ij}(g)) =   f(v) \cdot (\lambda_i - \lambda_j)  = f(v) \cdot 0 = 0 \] for all $v$.
 This shows that $\rho_{ij}(g) \in k$, and so $g \in T(k)$. We have thus shown that  $U  \cap G(K) = T(k)$, as desired.
\end{proof}
\vskip 5pt

The following is then obvious.

\begin{cor}\label{parahcor}  Under the hypotheses of Proposition \ref{parah}, one can choose Iwahori subgroups $I_0 \subset U_0$ and
$I_\infty \subset U_\infty$ such that, if $I_\infty(p) \subset I_\infty$ is the maximal pro-$p$ subgroup, then
$$G(K) \cap \prod_{x \neq 0,\infty} U_x \times I_\infty(p)\times I_0 = \{1\}.$$
\end{cor}
\vskip 5pt

\begin{remark}\label{hyper}  Suppose $U_z = G(\CO_z)$ for $z \in \Gm(k)$.   Then in Proposition \ref{parah}, we can take $U_0 = I^+$ and $U_\infty = I^-$ to be respectively the upper and lower Iwahori subgroups; by this we mean that $U_0$ (resp. $U_\infty$) is the Iwahori subgroup corresponding to the positive (resp. negative) root system corresponding to our fixed Borel subgroup $B$.  Then 
$$G(K)\cap U \subset G(K) \cap \prod_x G(\CO_x).$$
The right-hand side is the group of global maps from $\PP^1_k$ to the affine group $G_k$.   Any such map must be constant, because $\PP^1$ is projective, and thus must belong to $G(k)$ because it is defined over $k$.   Thus 
$$G(K)\cap U = (U_0 \times U_\infty)\cap G(k) = T(k).$$
\end{remark}
\vskip 5pt

The open compact subgroups of $G(\mathbb{A}_K)$ built in this section will be used in the next section to constuct some Poincar\'e series.
\vskip 10pt

\section{ Proof of Theorem \ref{mainthm}(i)}  \label{S:mainthm}
This section is devoted to the proof of Theorem \ref{mainthm}(i). 
For the convenience of the reader, we repeat the statement here:
\vskip 5pt

\begin{thm}\label{mainthm7}
Suppose $F$ is the non-archimedean local field $k(\!(t)\!)$, where $k = \Fq$ is a finite field of order $q = p^m$ for some prime $p$.    Let $G$ be an unramified connected reductive  group over $F$ which is not a torus.  Assume $q > 5$ and  $\pi$ is an irreducible  supercuspidal representation of $G(F)$ that can be obtained as the induction of a representation of a compact open (modulo center) subgroup $U \subset G(F)$:
$$\pi \isoarrow {\rm c-Ind}_U^G(F) \tau$$
for some (irreducible) smooth representation $\tau$ of $U$ with coefficients in $\Qlb$. 
 
 \vskip 5pt
 
 Then we have:
 \[  \text{ $\CL^{ss}(\pi)$ is pure}  \Longrightarrow  \text{$f \circ \CL^{ss}(\pi)= f\circ \CL(\pi)$ is ramified,} \]
 where $f: {^L}G \longrightarrow {^L}G / Z(\hat{G})$ is the natural projection map. 
 \end{thm}

\vskip 5pt

Because ${^L}G / Z(\hat{G})$ is the L-group of the simply-connected cover $G_{sc}$ of the derived group $G_{der}$ of $G$,  Theorem \ref{parab}(vi) implies (by applying it to the natural maps $G_{sc} \to G_{der} \to G$) that  it suffices to prove the above theorem under the additional hypothesis that $G$ is $k$-simple and simply-connected.

%\subsection{Reduction to absolutely simple simply-connected case}
%Let us first explain how one can reduce the proof of the above theorem to the case when $G$ is absolutely simple and simply connected.
%\vskip 5pt
%
%\begin{itemize}
%\item Because ${^L}G / Z(\hat{G})$ is the L-group of the simply-connected cover $G_{sc}$ of the derived group $G_{der}$ of $G$,  Theorem \ref{parab}(vi) implies (by applying it to the natural maps $G_{sc} \to G_{der} \to G$) that  it suffices to prove the above theorem under the additional hypothesis that $G$ is semisimple and {\it simply-connected}.
%\vskip 5pt
%
%\item With $G$ semisimple and simply-connected, we have:
%\[  G \cong \prod_i {\rm Res}_{F_i/F}(G_i) \]
%with each $G_i$ an absolutely simple simply-connected group over some finite extension $F_i$ of $F$.  By Theorem \ref{parab}(vi) again, it suffices to treat each factor separately, so that we may assume that $G = {\rm Res}_{E/F}(G')$ is $F$-simple and simply-connected.
%
%
%\vskip 5pt
%
%\item By the compatibility of the Genestier-Lafforgue construction with restriction of scalars, we may assume further that $G$ is absolutely simple and simply-connected.
%\end{itemize}
%
%
% 
%\vskip 5pt
%
%
%
%
% \vskip 5pt
%

\subsection{Purity and Ramification} \label{ram1}
We shall prove Theorem \ref{mainthm}(i) by a global argument. Hence, we assume the setup and notation from \S \ref{SS:unramified}. In particular, $K = k(\PP^1) = k(t)$ and $G$ is an unramified simply connected simple  group over $k$. 

Before stating this theorem, we choose a special character of $T$. The adjoint representation of $\hat{G}$ splits as a sum of irreducible representations, each corresponding to a simple factor, up to isogeny, of $G_{\overline{k}}$. The group $W_k$ acts on these irreducible representations. Let $\mathfrak v_1,\dots, \mathfrak v_n$ be an orbit of this action.  Let $\hat{\alpha}_i $ be the highest weight of $\mathfrak v_i$, a coroot of $G$, and let $\alpha_i$  be the corresponding root of $G$. Let $\alpha \in X^* (T) $ be given by $\alpha = \sum_{i=1}^n \alpha_i$. Then since the action of $W_i$ permutes the $\mathfrak v_i$s, it permutes the $\hat{\alpha}_i$s, and thus the $\alpha_i$s, so $\alpha$ is $W_k$-invariant.

 \vskip 5pt
 
 The following theorem is the key  step  in proving Theorem \ref{mainthm}(i).
  \vskip 5pt

\begin{thm}\label{notunramified}  Let  $F = k(\!(t)\!)$ with $k = \mathbb{F}_q$ and $q>3$, and let $\sigma$ be a supercuspidal representation of $G(F)$ (with $G$ unramified over $k$).   Let $K = k(\PP^1)$ and suppose there exists a cuspidal automorphic representation $\Pi$ of $G(\ad_K)$, with the following properties:
\begin{itemize}
\item[(a)]  At every $z \in \Gm(k) \subset \PP^1(k)$, $\Pi_z \isoarrow \sigma$.
\vskip 5pt

\item[(b)]  The representation $\Pi_\infty$ has a vector invariant under the pro-$p$ Iwahori subgroup $I_{\infty}(p)$ in Corollary \ref{parahcor}.
\vskip 5pt

\item[(c)]  The representation $\Pi_0$ is a constituent of a principal series $Ind_{B(F)}^{G(F)} \chi$, where
$\chi$ is a tame (i.e. depth 0) character of $T(F)$, whose restriction  $\chi_k$ to the subgroup $T(k) \subset T(F)$ arises  in the following way:
\[ \begin{CD} 
\chi_k:  T(k) @>\alpha>>  k^{\times} @>\mu>> \overline{\mathbb{Q}}_{\ell}^{\times} \end{CD} \]
  where $\mu$ is a faithful character of $k^{\times}$.
\vskip 5pt

\item[(d)]   The local components $\Pi_y$ of $\Pi$ for all $y \notin \PP^1(k)$ are $G(\mathcal{O}_y)$-unramified.
\end{itemize}
Then
\[ \text{$\CL^{ss}(\sigma)$ is pure} \Longrightarrow \text{$\CL^{ss}(\sigma)$ is ramified.} \]
 \end{thm}

\begin{proof}   Assume  $\CL^{ss}(\sigma)$ is unramified; we shall derive a contradiction.
\vskip 5pt

Applying Lafforgue, let 
\[  \CL(\Pi):  {\rm Gal}(K^{sep}/K) \longrightarrow {^L}G = \hG \rtimes W_k \]
 be the global Galois representation associated to $\Pi$  as in Theorem \ref{paramVL}.  By property (d), $\CL(\Pi)_x$ is unramified for all places $x$ of $\PP^1$ outside of $\PP^1(k)$.
 On the other hand, by property (b) and (c), it follows by Theorem 
 \ref{parab} (iv) that $\CL(\Pi)^{ss}_0$ and $\CL(\Pi)^{ss}_{\infty}$ are tamely ramified. 
For $z \in \Gm(k)$, with decomposition group
 $\Gamma_z \subset \pi_1(\PP^1 \setminus |\PP^1(k)|)$, it follows by property (a) and Theorem \ref{paramVL}(iii) that
\begin{equation}\label{locz}
\CL(\Pi)^{ss}_z \isoarrow \CL^{ss}(\sigma).
\end{equation} 
By our hypothesis, $\CL(\Pi)^{ss}_z$ is thus unramified for all $z \in \mathbb{G}_m(k) = k^{\times}$. 
 
\vskip 5pt
 
We shall now examine  in greater detail the restriction of $\CL(\Pi)$ to the local decomposition group $\Gamma_0$ at $0 \in k$. 
 By property (c) and Theorem  \ref{parab} (iv), one has
\[ \CL(\Pi)^{ss}_0 : \Gamma_0 = {\rm Gal}(F^{sep}/F) \longrightarrow {^L}T \longrightarrow {^L}G \] 
where the first map is the L-parameter of the character $\chi$ of $T(F)$.  By the discussion in \cite[\S 4.3]{DR}, one has
\[   {\rm Hom}(T(k), \overline{\mathbb{Q}}_l^{\times})  \cong {\rm Hom}_{{\rm Frob}}(I_t , \hat{T}) \cong {\rm Hom}_{{\rm Frob}}(k_n^{\times}, \hat{T})  \]
where $k_n = \mathbb{F}_{q^n}$ is the splitting field of $T$ over $k$, $I_t$ is the tame inertia group of $\Gamma_0$ and ${\rm Hom}_{{\rm Frob}}$ stands for Frobenius-equivariant homomorphisms. Hence, the restriction of $\CL(\Pi)^{ss}_0$ to the inertia group (which factors to the tame inertia group)
determines and is determined by the restriction $\chi_k$ of $\chi$ to $T(k)$. The special form of $\chi_k$ given in property (c)  and the fact that $\alpha$ is fixed by ${\rm Frob}$ thus imply that the restriction of $ \CL(\Pi)^{ss}_0 $ to $I_t$ factors as:
 \begin{equation} \label{E:alphaIt}
   \begin{CD}
\CL(\Pi)^{ss}_0|_{I_t} :   I_t @>>> k^{\times} @>\mu>> \overline{\mathbb{Q}}_l^{\times}@>\alpha>> \hat{T} @>>> {^L}G. \end{CD}   \end{equation}
\vskip 5pt

Let  $V= \otimes_{i=1}^n \mathfrak v_i$ be the representation of  ${^L}G = \hat{G} \rtimes W_k$ where $\hat{G}$ acts on each $\mathfrak v_i$ and $W_i$ permutes the factors.
%Let element $\alpha^{\vee} \in X_*(T) = X^*(\hat{T})$ be given by $\alpha^\vee =\sum_{i=1}^n\alpha_i^\vee$, in other words, the highest weight of the representation $V$.
The composite $V \circ \CL(\Pi)^{ss}$ gives rise to a 
(semisimple) local system $\CV(\Pi,V)$ on $\PP^1 \setminus |\PP^1(k)|$ over $k$. 
 However, since $ \CL(\Pi)^{ss}_z$ is   unramified, the purity of $\CL^{ss}(\sigma)$ and Theorem \ref{pure} then imply that $\CV(\Pi,V)$ extends to a local system on $\Gm$ over $k$.   
 \vskip 5pt
 
 The representation $V$ is obtained as a tensor product of the representations $\mathfrak v_i$ of $\hat{G}$, which do not extend to representations of $\hat{G} \rtimes W_k$, but do extend to representations of $\hat{G} \rtimes W_{k_n}$ for $k_n$ the degree $n$ extension, since the stabilizer of $\mathfrak v_i$ in the action of $W_k$ on irreducible factors of the adjoint representation of $\hat{G}$ is $W_{k_n}$. Thus, we similarly have that the composite $\mathfrak v_i \circ \CL(\Pi)^{ss}\mid_{ {\rm Gal} (K^{sep} / k_n( t))}$ gives rise to a 
(semisimple) local system $\CV(\Pi,\mathfrak v_i)$ on $\PP^1 \setminus |\PP^1(k)|$ over $k_n$, which extends to a local system on $\Gm$ over $k_n$.

 The above discussion shows that the local system $\CV(\Pi,\mathfrak v_i)$ on $\GG_m$ over $k_n$ satisfies the hypotheses of Corollary \ref{abelian}. In particular:
 \begin{itemize}
 \item[(a)]   the local monodromy representations at $0$ and $\infty$ are tamely ramified;
 \item[(b)] the local monodromy representation at $0$ has abelian image. 
 \end{itemize}
 Corollary \ref{abelian} thus implies that $\CV(\Pi,\mathfrak v_i )$ breaks up as the sum of rank 1 local systems $\mathcal L_1, \dots, \mathcal L_N$ over $k_n$.   
 \vskip 5pt
 
  Applying Lemma \ref{Frobenius-trace} to the $\mathcal L_i$'s, and a standard formula for the trace on a tensor product representation of an operator that permutes the tensor factors, one sees that for each $z \in \mathbb{G}_m(k)$,
  \begin{equation}\label{trace-decomposition} \operatorname{tr}(\operatorname{Frob}_z, \CV(\Pi,V)) =   \operatorname{tr}(\operatorname{Frob}_z^n, \CV(\Pi,\mathfrak v_1)) = \sum_{j=1}^N \operatorname{tr}(\operatorname{Frob}_z, \mathcal L_j) = \sum_{j=1}^n \operatorname{tr}(\operatorname{Frob}_1^n, \mathcal L_j) \cdot \operatorname{tr}(m_z , \mathcal L_j).\end{equation}
 Here it is crucial that we consider the tensor induction instead of the usual induction because the trace of  $\operatorname{Frob}_z$ on the usual induction would be identically zero as soon as $n>1$ and so studying this trace would not be helpful. We take $i=1$ without loss of generality because $\operatorname{tr}(\operatorname{Frob}_z^n, \CV(\Pi,\mathfrak v_i))$ is independent of $i$. 
 
 Let us now evaluate $ \operatorname{tr}(m_z , \mathcal L_j)$. By (\ref{E:alphaIt}), and observing that the eigenvalues of $\hat{T}$ on $\mathfrak v_1 $ are the roots of $(\hG, \hat{T})$ contained in the simple factor corresponding to $\mathfrak v_1$, we see that
  \begin{equation} \label{E:tr}
    \operatorname{tr}(m_z , \mathcal L_j ) = \mu(z)^{ \langle \alpha, \beta_j \rangle }
    \end{equation}
 where  $\beta_j$ is either a root of $\hG$ (or equivalently a coroot of $G$) contained in $\mathfrak v_1$ or the zero element in $X^*(\hat{T})$. 
\vskip 5pt

Now  note that:
\vskip 5pt

\begin{itemize}
\item[(a)] $\mu$ is a character of order $q-1$ of $k^{\times}$;
\item[(b)]  the function
 \[  z \mapsto \operatorname{tr}(\operatorname{Frob}_z, \CV(\Pi,{V})) \]
 is a constant function of $z$, since the local representation at each $z \in \mathbb{G}_m(k)$ is just $\CL^{ss}(\sigma)$. 
\end{itemize}
By orthogonality of characters (of $k^{\times}$) and \eqref{trace-decomposition}, we thus have
\vskip 5pt

\begin{align} \label{Fourier-analysis}
 0 &= \sum_{z \in k^\times} \mu(z)^{-2} \cdot \operatorname{tr}(\operatorname{Frob}_z, \CV(\Pi,{V})) \notag \\
 &=   \sum_{z \in k^\times} \mu(z)^{-2} \cdot \sum_{j=1}^N \operatorname{tr}(\operatorname{Frob}_1^n, \mathcal L_j) \cdot \mu(z)^{ \langle \alpha, \beta_j \rangle } \notag \\
 &=  \sum_{j=1}^N  \operatorname{tr}(\operatorname{Frob}_1^n, \mathcal L_j)  \cdot  \sum_{z \in k^\times} \mu(z)^{-2} \cdot \mu(z)^{ \langle \alpha, \beta_j \rangle }  \notag \\
 &= (q-1) \cdot \sum_{\substack{ j: \\ \langle \alpha, \beta_j \rangle \equiv 2 \mod (q-1) }} \operatorname{tr}(\operatorname{Frob}_1^n, \mathcal L_j) .
 \end{align}
\vskip 5pt

\noindent Furthermore we we have \[\langle \alpha, \beta_j \rangle = \langle \sum_{i=1}^n \alpha_i, \beta_j \rangle =\sum_{i=1}^n \langle \alpha_i,\beta_j\rangle = \langle \alpha_1,\beta_j \rangle \] since for $i \neq 1$, $\alpha_i$ is a root of a different simple factor from $\beta_j$ and thus is orthogonal to $\beta_j$.
 
Now one of the $\beta_j$'s is the highest weight $\alpha_1^\vee$  of $\mathfrak v_1$, for which one has $\langle \alpha, \alpha_1^{\vee} \rangle = \langle \alpha_1,\alpha_1^\vee\rangle=2$. For all other $\beta_j$, one has
\[  -2 \leq \langle \alpha, \beta_j \rangle  < 2. \]
Hence, since $\langle \alpha, \beta_j \rangle \in \mathbb{Z}$, we have 
\[   \langle \alpha, \beta_j \rangle \in \{-2,-1,0,1\}. \]
Since we have assumed $q > 5$, this implies that there is only one term  appearing in the sum in (\ref{Fourier-analysis}), namely the term corresponding to the highest root $\alpha_1^{\vee}$ of $\hG$.  
Hence, we deduce the desired contradiction:
\[  (q-1) \cdot \operatorname{tr}(\operatorname{Frob}_1^n, \mathcal L_j)   = 0 \]
which is impossible as $\mathcal L_j$ is a one-dimensional representation. 
\vskip 5pt

We have thus completed the proof of Theorem \ref{notunramified}.

\end{proof}

\vskip 5pt

\subsection{ Poincare series I}  \label{poincaresplit}
We continue with the proof of Theorem \ref{mainthm}(i), under the hypothesis that $G$ is semisimple and simply-connected.
Let  $F = k(\!(t)\!)$ and let $\sigma$ be a supercuspidal representation of $G(F)$.    We say $\sigma$
 contains an $\mathfrak{s}$-type, if there is an open compact subgroup $U^* \subset G(F)$ and an irreducible
 representation $\rho$ of $U^*$ such that 
 \begin{equation}\label{cind}  \sigma \isoarrow c-Ind_{U^*}^{G(F)} \rho.
 \end{equation}

\begin{thm}\label{poin} Let $K = k(\PP^1)$.   Let $F = k(\!(t)\!)$ and let $\sigma$ be a supercuspidal representation of $G(F)$.    Suppose 
\begin{equation}\label{type}
\sigma \text{ contains an } \mathfrak{s}-\text{type} \,\,  (U^*,\rho).  
\end{equation}
Then there exists a cuspidal automorphic representation $\Pi$ of $G(K)$
 satisfying the conditions (a)-(d)  of Theorem \ref{notunramified}.
  \end{thm}
  
  \begin{proof}  The construction follows the method of \cite{GLo}.   Hypothesis \eqref{type} implies that $\sigma$ has a matrix coefficient $\varphi^*$ with support
  in $U^*$, with the property that $\varphi^*(1) = 1$.  
  Using this $U^*$, let $U = \prod_x U_x \subset G(\ad_K)$ be the open compact subgroup constructed  in Proposition \ref{parah}.  
  Define a function $\varphi = \prod_x \varphi_x:  G(\ad_K) \ra \CC$ as follows:
  \begin{itemize}
\item[(a)]  At every $z \in \Gm(k) \subset \PP^1(k)$, $\varphi_z = \varphi^*$;
\vskip 5pt
 
\item[(b)]  $\varphi_\infty$ is the characteristic function of the group $I_\infty(p)$ of Corollary \ref{parahcor};
\vskip 5pt
\item[(c)]  $\varphi_0$ is the character $\chi_k:  I_0/I_0(p) \cong T(k) \ra \CC^\times$, where $\chi_k$ is as in the proof of 
Theorem \ref{notunramified};
\vskip 5pt
\item[(d)]  For $x \notin |\PP^1(k)|$, $\varphi_x$ is the characteristic function of  the hyperspecial maximal compact subgroup $G(\CO_x)$;
\end{itemize}
 \vskip 5pt
 
 \noindent Now define the Poincar\'e series $P_\varphi:  G(K)\backslash G(\ad_K) \ra \CC$:
 $$P_\varphi(g) = \sum_{\gamma \in G(K)} \varphi(\gamma\cdot g).$$
 The sum converges absolutely, as in \cite{GLo}.
 It follows from Corollary \ref{parahcor} that
 $$P_\varphi(1) = \prod_{z \in \Gm(k)} \varphi^*(1) = 1.$$
 In particular, $P_\varphi \neq 0$. The spectral decomposition of $P_{\varphi}$ yields the desired cuspidal automorphic representation $\Pi$.
  
  \end{proof}
  
  \begin{cor}\label{maincor}  Let $G$ be an unramified reductive group over a non-archimedean local field $F$ of characteristic $p$, with constant field $k$ of order $q > 5$.  Suppose that $G$ is not a torus and $p$  does not divide the order of the Weyl group of $G$. Then  if $\sigma$ is pure and $q > 5$,  $f \circ \CL^{ss}(\sigma)$ is ramified (where $f: {^L}G \to {^L}G/Z(\hat{G})$). In particular, Theorem \ref{mainthm}(i) holds.

%Let $K = k(T) = k(\PP^1)$ and let $\sigma$ be a supercuspidal representation of $G(F)$.   Then there exists a cuspidal automorphic representation $\Pi$ of $G(K)$
 %satisfying the conditions (a)-(d) and (f) of Theorem \ref{notunramified}.%, such that $\chi_k$.  

 \end{cor}

\begin{proof}  
As mentioned at the beginning of the section, we may assume $G$ is simply-connected. 
Thanks to the results of \cite{FiPreprint}, the supercuspidal representation $\sigma$ of $G(F)$ contains an $\mathfrak{s}$-type by our hypothesis on $p$.  The claim now follows immediately from Theorems \ref{notunramified} and \ref{poin}.  
\end{proof}
\vskip 5pt

\begin{cor}\label{sph}  Let $F = k((t))$ and let $\sigma$ be an irreducible representation of $G(F)$, with $G$ unramified reductive over $F$.  
Suppose the Genestier-Lafforgue parameter $\CL^{\ss}(\sigma)$ is pure and unramified.  Suppose every supercuspidal representation of every Levi factor of $G(F)$ contains an $\mathfrak{s}$-type; this is the case in particular if $p$ does not divide the order of the Weyl group of $G$.  
Then $\sigma$ is an irreducible constituent of an unramified principal series representation of $G(F)$.

\end{cor}

\begin{proof}  In any case $\sigma$ is an irreducible constituent of a representation of the form $Ind_{P(F)}^{G(F)} \tau$ for some parabolic subgroup $P(F) = L(F)\cdot N(F) \subset G(F)$, where $L(F)$ is a Levi subgroup of $P(F)$, $N(F)$ its unipotent radical,
and $\tau$ is a supercuspidal representation of $L(F)$.  Since the Genestier-Lafforgue parametrization is compatible with parabolic induction by Theorem \ref{parab}, it follows that $\CL^{ss}(\tau)$ is unramified.  Since $\tau$ is supercuspidal, it follows from Theorems \ref{poin} and \ref{notunramified} (or more simply Corollary \ref{maincor})  that $L$ must be a torus and $\tau$ an unramified character of $L(F)$. 
\end{proof}

%More generally, if $\sigma$ is any pure irreducible representation of $G(F)$ such that $\CL^{ss}(\sigma)$ is unramified then $\sigma$ is an irreducible constituent of an unramified principal series representation.
%\end{cor}

%hen $F$ is isomorphic to $k(\!(t)\!)$, which is the completion of at the point $v=0$ of the function field $K = k(\mathbb{P}^1) \simeq k(t)$. Denote by $\mathcal{O}_F$ the ring of integers of $F$, and by $\mathfrak{p}_F = (\varpi_F)$, its maximal prime ideal with uniformizer $\varpi_F$.

%In this setting, we can take $\mathcal{K} = G(\mathcal{O}_F)$ for the hyperspecial maximal compact open subgroup of $G(F)$. Thanks to the results of \cite{FiPreprint}, a supercuspidal representation $\pi$ of $G(F)$ contains an $\mathfrak{s}$-type, $(\mathcal{K},\rho)$. We can generalize the argument of \cite{HeLo} for ${\rm GL}(n)$.

%\begin{thm}\label{mainthmsplit}
%Given a unitary supercuspidal representation $\pi$ of $G(F)$, there is a cuspidal automorphic representation $\Pi$ of $G(\mathbb{A}_K)$, such that
%\begin{itemize}
%   \item[(i)] $\pi \simeq \Pi_0$;
%   \item[(ii)] $\Pi_v$ is unramified at $v \notin \{ 0, \infty \}$;
%   \item[(iii)] $\Pi_\infty$ is tamely ramified with $\Pi_\infty^{I_-(p)} \neq 0$.
%\end{itemize}
%\end{thm}

%{\color{blue} Now, here are a few changes or variations that can be done here, depending on what is needed\ldots

%\begin{prop}\label{generic} Under the hypotheses of Corollary \ref{maincor}, suppose $\pi$ is generic.  Then we may assume $\Pi$ to be generic as well.
%\end{prop}

%\begin{proof}

%\end{proof}

\section{Proof of Theorem \ref{mainthm}(ii): Wild ramification}\label{wildram}

The goal of this section is to prove Theorem \ref{mainthm}(ii). For the convenience of the reader, we restate the result here:
\vskip 5pt
\begin{thm}\label{wild} 
Let $G$ be an unramified reductive group over a local field $F= k((t))$ of equal characteristic $p$,  and $J$ an open subgroup of a parahoric subgroup $G_a$. Let $\sigma$ be a representation of $J$ such that $\pi = \operatorname{c-Ind}_J^{G( F)} \sigma$ is irreducible. 
\vskip 5pt

If $J$ is sufficiently small and $|k| = q>2$, then    $f\circ \CL^{ss}(\pi)$ (and hence $f \circ \CL(\pi)$) is wildly ramified (where $f: {^L}G \to {^L}G/Z(\hat{G})$).
\end{thm}

\vskip 5pt

\subsection{Sufficiently small subgroups}
Let us first give the precise definition of being ``sufficiently small", recalling that $G$ is defined over $k$.
\vskip 5pt

\begin{defn}
(i)  A finite subgroup $H$ of $G(k)$ is
 is \emph{sufficiently small} if there is a Borel $k$-subgroup $B$ of $G$ such that $B(k) \cap H$ is trivial.

\vskip 5pt

(ii)  For  a parahoric subgroup $G_a$ of $G(F)$, we say $J \subseteq G_a$ is \emph{sufficiently small} if its projection to the reductive group $G_a/G_{a,0+}$ is sufficiently small. 
\end{defn}

For example, the maximal unipotent subgroup of a Borel subgroup is a
 a sufficiently small subgroup of $G(k)$  (by the Bruhat decomposition) and  $G_{a,0+}$ is a sufficiently small subgroup of $G(F)$.
\vskip 5pt

%\begin{remark} This condition could likely be weakened and the below argument would still work.  In principle as long as $B(\mathbb F_q) \cap H$ is contained in a proper algebraic subgroup of the maximal torus of $B$, this should be fine for $q$ sufficiently large. \end{remark} 

    \vskip 5pt
    
 \subsection{ An open  compact subgroup of $G(\mathbb{A}_K)$)  }  
  As for Theorem \ref{mainthm}(i), we may and shall assume that $G$ is simply-connected when proving Theorem \ref{wild}. 
The proof will be via a global argument. Hence we will be working over $K = k(\PP^1)$ and make use of the notation of \S \ref{SS:unramified}.
We will need to construct appropriate globalizations of $\pi$ by Poincar\'e series and the purpose of this section is to construct appropriate open compact subgroups of $G(\mathbb{A}_K)$. 
 \vskip 5pt
 
 Recall from \S \ref{SS:buildings} that for $z \in \PP^1(k)$, the apartments  $\mathcal{A}(S_z)$ in the buildings $\mathcal{B}(G_z)$  can be naturally identified with $X_*(S) \otimes \RR$ and hence with each other. 
 Without loss of generality, we may assume that the parahoric subgroup $G_a$ in Theorem \ref{wild}  is associated to a point $a \in \mathcal{A}(S_0) = X_*(S) \otimes \RR$.
 Let $b$ be a small perturbation of the point $-a$ in $\mathcal{A}(S_{\infty})$. Specifically, we choose $b$ so that $a+b$ is a small rational multiple of a cocharacter which lies in the interior of the positive Weyl chamber of $X_*(S) \otimes \RR$.  Because $a+b$ does not lie on the walls of the Weyl chamber, $b$ does not lie on any wall of the apartment that also contains $a$, and because $a+b$ is small, $b$ does not lie on any wall of the apartment that doesn't contain $a$.  So $b$ lies in the interior of a Weyl alcove and thus gives rise to an Iwahori subgroup $G_b \subset G(K_{\infty})$.
 \vskip 5pt
 
 For each place $v$ of $K$, set
 \[  f(v) = \begin{cases}
  a, \text{  if $v = 0$;} \\
  b, \text{  if $v = \infty$;} \\
  0, \text{   for all other $v$.} \end{cases} \]
 Let 
 \[  C  = \prod_vC_{f(v)} \subset G(\mathbb{A}_K) \]
 be the associated  compact open subgroup of $G( \mathbb A_K ) $, so that for $v \notin \{0, \infty\}$, $C_v$ is the standard hyperspecial maximal compact subgroup 
 $G(\mathcal{O}_v)$.
 \vskip 5pt
 
  We now note:
 \begin{lemma}
 The natural map
  \[ C \cap G(K) \to C_a  \to G_a \to G_a/G_{a,0+}\]  
   is injective with image contained in a Borel subgroup of $G_a/ G_{a,0+}$.
    \end{lemma}
 
 \begin{proof}
   Let $\lambda \colon \mathbb  G_m \to S$ be a cocharacter which is a negative multiple of $a+b$.   If we consider the image of $\lambda$ in $G_{a}/G_{a,0+}$ continues to  define a cocharacter of the maximal split torus of $G_{a}/G_{a, 0+}$. Since $\lambda$  lies in the interior of a Weyl chamber, it defines a Borel subgroup $B_{\lambda}$ over $k$ (under the ``dynamical" definition of parabolic subgroups). We shall show that the image of $C \cap G(K)$ in $G_a/G_{a,0+}$ belongs to $B_{\lambda}$. For this, we need to show that for any $g \in C \cap G(K)$,
   \[ \text{$\lambda(t) \cdot g \cdot \lambda(t)^{-1}$ remains bounded as $t \to 0$.} \]
     \vskip 5pt
     
      To verify this, we  shall apply Lemma \ref{parahoric-coordinates-lemma} to the  faithful algebraic representation $\rho\colon G \to \operatorname{Aut}(V)$ of $G$ over $k$  that was fixed in \S \ref{SS:unramified}. Recall that $V$ is equipped  with a $k$-basis
  $e_1,\dots, e_n$ of $V$, where $e_i$ is an eigenvector of $S$ with eigencharacter $\lambda_i$.
  For any $g \in C \cap G(K)$ and $(i,j)$ such that $\rho_{ij}(g) \ne 0$, we deduce 
 by the product formula and Lemma \ref{parahoric-coordinates-lemma}  that     
  \begin{equation}\label{second-pf-equation} 
  0 = \sum_v v (\rho_{ij}(g)) \geq   \sum_v  \lceil f(v) \cdot (\lambda_i -\lambda_j) \rceil  =  \lceil a \cdot (\lambda_i -\lambda_j) \rceil + \lceil b \cdot (\lambda_i -\lambda_j ) \rceil .\end{equation}
  Now  since $b$ is a small perturbation of $-a$, \eqref{second-pf-equation} can only be satisfied if $ a \cdot (\lambda_i - \lambda_j) $ is an integer and $b \cdot (\lambda_i -\lambda_j) \leq -a \cdot (\lambda_i - \lambda_j)$.  
  \vskip 5pt
  
  With $\lambda$ chosen as in the beginning of the proof, we thus deduce that
  for $g\in C \cap G(K)$ such that  $\rho_{ij}(g) \neq 0$,   we have $(a+b) \cdot  (\lambda_i - \lambda_j ) \leq 0$  and thus 
  \[  \lambda\cdot (\lambda_i -\lambda_j) \geq 0. \] 
  It follows that  $\lambda(t) g \lambda (t)^{-1}$ is bounded as $t \to 0$, since all the nonvanishing coordinates  $\rho_{ij}$ of $g$ are eigenvectors under the conjugation action by $\lambda(t)$ with eigenvalue a negative power of $t$. This implies that the image of $g$ in 
 $G_{a}/G_{a,0+}$ lies in the Borel subgroup $B_{\lambda}$, as desired.

  \vskip 5pt
  
  It remains to show that if $g \in C \cap G(K)$ has trivial image in $G_a/ G_{a, 0+}$, then $g$ is the identity element. 
  If $\rho_{ij}(g) \neq 0$, then because $ a \cdot (\lambda_i - \lambda_j) $ is an integer and $b \cdot (\lambda_i -\lambda_j) \leq -a \cdot (\lambda_i - \lambda_j)$, the right side of \eqref{second-pf-equation} is $0$. So every inequality is sharp, i.e.
  \[  v_0(\rho_{ij}(g))  = a \cdot(\lambda_i - \lambda_j). \]
  On the other hand, by Lemma \ref{parahoric-coordinates-lemma}(ii), if $g \in G_{a, 0+}$, then 
  \[  v_0(\rho_{ij}(g)) - \delta_{ij}) >    a \cdot (\lambda_i -\lambda_j). \]  
  Thus, if $i \neq j$, we obtain a contradiction from the assumption $\rho_{ij}(g) \neq 0$. Hence, $\rho_{ij}(g) =0$ for $i \ne j$. On the other hand, if $i=j$, then $\rho_{ij}(g) \in \mathcal{O}_v$ for every place of $K$, and thus lies in $k$, but since $v_0(\rho_{ii}(g) -1)  > 0$, we see that $\rho_{ii}(g) =1$.
   Thus, if $g \in G_{0,x+}$, then $\rho(g)$ is the identity, so that $g$ is the identity since $\rho$ is faithful.

 \end{proof}

 \subsection{ Poincar\'e series II}
 Using the above lemma, we will construct two different Poincar\'e series for $G$ which are unramified outside $\{0,\infty\}$ and 
 whose local representation at $0$ is $\pi$, but with different local representations at $\infty$.  
 
 \vskip 5pt
 
 Let $B_{\lambda}$ be a Borel subgroup of $G_a/G_{a,0+}$  containing the image of $C$. Let $H$ be the image of $J$ in $G_a/G_{a,0+}$. Since we are assuming that $J$ is sufficiently small,  there exists $g$ in $G_a/G_{a,0+}$ with $J \cap g B_{\lambda}  g^{-1}$ trivial.
  By conjugating $J$ by a lift of $g$, we may assume $g=1$. Set
 \[  C' = J \times \prod_{v \ne 0} C_v  \subset C \subset  G(\mathbb{A}_K). \]
 Then the previous lemma implies that natural map $C' \cap G(K) \to G_a / G_{a,0+}$ is injective with image contained in $H \cap B =\{1\}$. Hence, $C' \cap G(K)$ is trivial.
\vskip 5pt

Recall that $G_y \subset G(K_{\infty})$ is an Iwahori subgroup, so that  the quotient $G_{y} / G_{y,0+}$ is a torus  of rank $r \geq \dim S$, hence has at least $(q-1)^r$ points. It thus has a non-trivial character $\chi$ because $q > 2$.  We can view $\sigma$ and $\sigma \otimes \chi$ as representations of $C'$ by projection onto the components at $0$ and $\infty$.
 \vskip 5pt
 Now apply the Poincar\'e series construction to obtain:
 \vskip 5pt
 
 \begin{itemize}
 \item an automorphic  function $f_1$  created from  a matrix coefficient of $\sigma$, 
 \item an automorphic function  $f_2$   created from a matrix coefficient of  $\sigma \otimes \chi$.
 \end{itemize}
 These Poincar\'e series are nonzero because $C' \cap G(K)$ is trivial, so that the sum defining the series has only one nonzero term.
 From the spectral expansion of $f_1$ and $f_2$, we obtain two cuspidal automorphic representations $\Pi_1$ and $\Pi_2$ which are unramified outside $\{0,\infty\}$ and such that
 \begin{itemize}
 \item at $0$, the local components of $\Pi_1$ and $\Pi_2$ are both isomorphic to $\pi$;
 \item at $\infty$, $\Pi_{1,\infty}$ has nonzero Iwahori-fixed vectors, whereas $\Pi_{2,\infty}$ doesn't. Indeed, both these representations are subquotients of principal series representations induced from characters of $T(K)$ whose restrictions to $T(k)$ are $1$ and $\chi$ respectively. 
 \end{itemize}

 \vskip 5pt
 
 \subsection{Proof of Theorem \ref{wild}}
 We can now conclude the proof of Theorem \ref{wild}.
  Consider the global parameters $\mathcal{L}(\Pi_1)$ and $\mathcal{L}(\Pi_2)$ associated to $\Pi_1$ and $\Pi_2$ above. These are both unramified outside $\{0, \infty\}$. Moreover, 
 \[ \mathcal{L}(\Pi_1)^{ss}_0  \cong \mathcal{L}(\Pi_2)^{ss}_0, \quad \text{but} \quad \mathcal{L}(\Pi_1)^{ss}_{\infty}  \ncong \mathcal{L}(\Pi_2)^{ss}_{\infty}. \]
Indeed, at $\infty$, the associated local Genestier-Lafforgue parameters are tame, and their restrictions to the tame inertia group at $\infty$ correspond under class field theory to $1$ and $\chi \ne 1$ (up to semisimplification).   
 
 \vskip 5pt

  Now suppose, for the sake of contradiction, that  the local Genestier-Lafforgue parameter of $\pi$ is tame.   Thus $\mathcal L(\Pi_1)$ and $\mathcal L(\Pi_2)$ factor through the tame fundamental group of $\mathbb G_{m, k}$. In the tame fundamental group of $\mathbb G_{m, k}$, the inertia subgroups at $0$ and $\infty$ are the same subgroup, both being equal to the tame geometric fundamental group. Since $\mathcal L(\Pi_1)$ and $\mathcal L(\Pi_2)$ are isomorphic when restricted to the tame inertia group at $0$ and semisimplified, they must  also be isomorphic when restricted to the tame inertia group at $\infty$ and semisimplified.  This contradicts what we showed above:  that $\mathcal L(\Pi_1)$ is trivial while $\mathcal L(\Pi_2)$ is nontrivial when restricted to the tame inertia group at $\infty$ and semisimplified.
 This gives the desired contradiction and completes the proof of Theorem \ref{wild}.
 \vskip 10pt

  \subsection{A special case}
 When the group $J$ is contained in the principal congruence subgroup of a hyperspecial maximal compact, there is a shorter argument, though it requires the purity of $\CL^{ss}(\pi)$ as a hypothesis and appeals to Theorem \ref{notunramified}.
\vskip 5pt
\begin{prop}\label{Gammap} Under the hypotheses of Proposition \ref{wild}, suppose $G_a$ is a hyperspecial maximal compact and
$J$ is contained its principal congruence subgroup $G_{a,0+}$.  Let $\pi$ be a pure supercuspidal representation of $G(K)$
compactly induced from an irreducible representation $\sigma$ of $J$.  Then $\mathcal{L}^{ss}(\pi)$ is wildly ramified.
\end{prop}

\begin{proof}  We construct a Poincar\'e series as in Theorem \ref{poin} using a function $\varphi$ which at $\infty$ is a matrix coefficient
of $\sigma$ and is a characteristic function of the standard  hyperspecial maximal compact $U_v = G(\mathcal{O}_v)$ at all other places $v$.  Now the support of $\varphi_\infty$ is contained in $J$, and $J \times\otimes_{v \neq \infty}U_v \cap G(K) = \{1\}$.  Thus the Poincar\'e series  $P_\varphi$ does not vanish at $1$, and the space of cusp forms it generates contains a cuspidal automorphic representation $\Pi$ that is unramified except at $\infty$.  So $\CL(\Pi)$ is a parameter with values in ${^L}G$ that is unramified away from $\infty$.  On the other hand, by Theorem \ref{notunramified} $\CL^{ss}(\pi)$ is ramified.  Since a tamely ramified local system on $\mathbb{A}^1$ is unramified, it follows that the ramification at $\infty$ must be wild. 
\end{proof}

 \vskip 10pt

\section{Questions}  \label{S:questions}

In this section, we  raise a number of natural  questions which are suggested by our results.
\vskip 5pt

\subsection{ Positive depth representations}
In the context of Theorem \ref{wild}, for a given reductive group $G$ over a local field $F$, one may ask for a better understanding (or even a classification) of all  supercuspidal representations satisfying the condition of Theorem \ref{wild}, i.e.that  can be induced from a sufficiently small open compact (modulo center) subgroup $J$. In particular, one may ask if they can be understood in the framework of J.K. Yu's construction of supercuspidal representations. For example, one might hope that if the smallest twisted Levi subgroup in the twisted Levi sequence in a Yu datum is a torus, then the resulting supercuspidal representation should satisfy  the condition of Theorem \ref{wild} or a slightly modified version of it.
\vskip 5pt

%\subsection{Purity of wild supercuspidals}

\subsection{ Examples of pure supercuspidals}
 Theorem \ref{mainthm}(i) has the serious condition that $\pi$ is a pure supercuspidal representation.
  In order for the theorem to be non-vacuous, it will be good to have some examples of pure supercuspidal representations.
  For example, the desiderata of the LLC suggests that if $\pi$ is a generic supercuspidal representation of a quasi-split $G$, then $\pi$ is pure. 
In the rest of this section, we  show this for certain generic supercuspidal representations of depth 0. 
\vskip 5pt

\subsection{Poincar\'e series III}
We begin by giving yet another construction of Poincar\'e series. Let us work in the context of \S \ref{SS:unramified} again, so that $G$ is an unramified semisimple group over $k$.
\vskip 5pt

 \vskip 5pt

\begin{prop}\label{generic0}
Let $\pi$ be a depth 0 generic supercuspidal representation of $G(F)$ (with $F = k((t))$) of the form
 \[  \pi \isoarrow c-Ind_{G(\mathcal{O}_F)}^{G(F)} \sigma, \]
where $\sigma$ is a generic cuspidal representation of the finite reductive group $G(k)$.   
There is a globally generic representation $\Pi$ of $G$ such that 
\vskip 5pt

\begin{itemize}
\item $\Pi_0 \cong \pi$ and $ \Pi_{\infty} \cong \pi^{\vee}$;
\item for all other places $v$ of $K = k(t)$, $\Pi_v$ is $G(\mathcal{O}_v)$-unramified.
\end{itemize}
\end{prop}

\begin{proof}
 Since $\sigma$ is a generic representation of $G(k)$, one can find a generic character 
\[  \psi: N(k) \longrightarrow \mathbb{C}^{\times} \]
such that the $(N(k), \psi)$-eigenspace of $\sigma$ is nonzero, in which case it is 1-dimensional. Let us fix a nonzero vector $w$ such that
\begin{equation} \label{E:Whit}
  \sigma(n) \cdot w = \psi(n) \cdot w  \quad \text{for all $n \in N(k)$.}\end{equation}
Fixing a $G(k)$-invariant inner product $\langle-, -\rangle$ on $\sigma$, we may assume that $\langle w,w \rangle =1$. 

\vskip 5pt

Now we may consider the function on $G(K_0)$ defined by
\[  f_0(g) =\begin{cases}
 \langle \sigma(g )\cdot w , w  \rangle , \quad \text{ if $g \in G(\mathcal{O}_0)$;} \\
 0, \quad \text{ if $g \notin G(\mathcal{O}_0)$.} \end{cases} \]
 This is a matrix coefficient of $\pi$, which is built out of a matrix coefficient of $\sigma$.
of $\sigma$ supported on $G(\mathcal{O}_0)$.
Likewise, we define a function on $G(K_{\infty})$ by:
\[  f_{\infty}(g) = \begin{cases}
\langle w, \sigma(g)\cdot  w \rangle, \quad \text{  if $g \in G(\mathcal{O}_{\infty})$;}  \\
0, \quad \text{  if $g \notin G(\mathcal{O}_{\infty})$,} \end{cases} \]
which is a matrix coefficient of $\overline{\pi} \cong \pi^{\vee}$.
Define a locally constant compactly supported function $f = \prod_v f_v$ on $G(\mathbb{A}_K)$ by requiring that $f_0$ and $f_{\infty}$ are as defined above and $f_v$ is the characteristic function of $G(\mathcal{O}_v)$ for all other $v$. Then consider the Poincar\'e series
\[  P_f(g) = \sum_{\gamma \in G(K)}  f(\gamma g) \]
Because of the support conditions on $f$, one has
\[  P_f(1) = \sum_{\gamma \in G(k)}  f_0(\gamma) \cdot f_{\infty}(\gamma) = \sum_{\gamma \in G(k)} | \langle \sigma(\gamma) \cdot w, w \rangle|^2. \]
This sum is certainly nonzero, so that $P_f$ is nonzero.
 \vskip 5pt

To see if $P_f$ is globally generic,  let us first construct an appropriate automorphic generic character 
\[  \Psi = \prod_v \Psi_v:  N(K) \backslash N(\mathbb{A}_K) \longrightarrow \mathbb{C}^{\times}. \]
Note that one has
 \[  N(k) \backslash N(\widehat{O}) \isoarrow N(K) \backslash N(\mathbb{A}_K), \]
 where $\widehat{O} = \prod_v \mathcal{O}_v$.
 We define $\Psi$ by requiring that
 \vskip 5pt
 \begin{itemize}
 \item $\Psi_v =1$ on $N(\mathcal{O}_v)$ for all $v \ne 0$ or $\infty$;
 \vskip 5pt
 \item the restriction of $\Psi_0$ to $N(\mathcal{O}_0)$ is obtained as
 \[ \begin{CD}
   \Psi_0 : N(\mathcal{O}_0) @>>> N(k) @>\psi>>  \mathcal{C}^{\times}
\end{CD} \] 
\vskip 5pt
\item likewise, the restriction of $\Psi_{\infty}$ to $N(\mathcal{O}_{\infty})$ is obtained as
\[ \begin{CD}
   \Psi_{\infty} : N(\mathcal{O}_{\infty}) @>>> N(k) @>\psi^{-1}>>  \mathcal{C}^{\times}
\end{CD} \] 
\end{itemize}
 Hence, for $n \in N(\widehat{O})$, one has
\[  \Psi(n) = \psi(n_0) \cdot \overline{\psi(n_{\infty})},\]
so that $\Psi$ is indeed trivial on the diagonally embedded  $N(k) \hookrightarrow N(\widehat{O})$.  
\vskip 5pt

We can now compute the $(N, \Psi)$-Whittaker-Fourier coefficient of $P_f$:
\begin{align}
  &\int_{N(K) \backslash N(\mathbb{A}_K)} \overline{\Psi(n)} \cdot P_f(n) \, dn  \notag \\
  =&\int_{N(k)\backslash N(\widehat{\mathcal{O})}} \overline{\Psi(n)} \cdot \sum_{\gamma \in G(K)} f(\gamma n) \, dn \notag  
   \end{align}
   For $n \in N(\widehat{O})$
  \[  \gamma n \in {\rm supp}(f) \Longrightarrow  \gamma \in G(\widehat{O}) \cap G(K) = G(k). \]
  Moreover, with $n \in N(\widehat{O})$, one deduces by (\ref{E:Whit}) that
  \[  f(\gamma n) =  \langle \sigma(\gamma\cdot   n_0) \cdot w, w \rangle \cdot   \langle w, \sigma(\gamma \cdot n_{\infty}) \cdot w \rangle = \psi(n_0) \cdot \overline{\psi(n_{\infty})} \cdot | \langle \sigma(\gamma) w,w \rangle|^2. \]
  Hence, one has
  \begin{align}
  &\int_{N(K) \backslash N(\mathbb{A}_K)} \overline{\Psi(n)} \cdot P_f(n) \, dn  \notag \\
  =& \int_{N(k)\backslash N(\widehat{\mathcal{O})}}  \overline{\psi(n_0)} \cdot \psi(n_{\infty}) \cdot   \sum_{\gamma \in G(k)} \psi(n_0) \cdot \overline{\psi(n_{\infty})} \cdot | \langle \sigma(\gamma) w,w \rangle|^2 \, dn. \notag \\
  =& {\rm Vol}(N(k) \backslash N(\widehat{O})) \cdot \sum_{\gamma \in G(k)} | \langle \sigma(\gamma) w,w \rangle|^2. \notag 
    \end{align}
 which is nonzero. Thus, $P_f$ is $(N,\Psi)$-generic.
 \end{proof}

\subsection{Purity of generic depth zero representations}
Now we can sketch our strategy for showing that a depth $0$ generic supercuspidal representation is pure.   
We use the following
\begin{ethm}[Dat-Lanard]\label{dlan}  Let $\pi \in \CA(G,F)$ be a depth $0$ supercuspidal representation of $G(F)$ (in the sense of \cite{DR}).  Then
the semisimple parameter $\CL^{ss}(\pi)$ is (at most) tamely ramified.
\end{ethm}

Using this, we can now show:
\vskip 5pt

\begin{thm}
Assume that $G$ is semisimple over $k$ and $p = {\rm char}(k)$ is a good prime for $G$. Then for any depth 0 generic supercuspidal representation of $G(F)$, the associated semisimple parameter $\CL^{ss}(\pi)$ is pure.
\end{thm}
\vskip 5pt

\begin{proof}
%For simplicity we assume $G$ to be semisimple.  Because we will be applying Theorem  \ref{T:Lparameter}  we need to assume $p$ is a good prime for $G$.
Consider  the global Lafforgue parameter  $\CL(\Pi)$ associated to  the globally generic cuspidal representation $\Pi$ of Proposition \ref{generic0}.  Since $\Pi$ is unramified outside
$0$ and $\infty$, it gives rise to a semisimple local system $\CL(\Pi,Ad)$ on $\Gm/k$, where $Ad$   is the adjoint representation of ${}^LG$.  
Expected Theorem \ref{dlan} implies that $\CL(\Pi,Ad)$ is tamely ramified.  It follows that, for a finite extension $k'/k$ of the constant field, the restriction
$\CL(\Pi,Ad)_{k'}$ of $\CL(\Pi,Ad)^{ss}$ to $\Gm/k'$ is a sum of $1$-dimensional tamely ramified $\ell$-adic local systems:
$$\CL(\Pi,Ad)_{k'} \isoarrow \oplus L_i$$
where each $L_i$ is pure of weight $w_i$.  Let $\Omega$ be the set of weights $w_i$ that occur.
\vskip 5pt

Now suppose $\CL^{ss}(\pi)$ is not pure.  Since $\pi$ belongs to the discrete series, $G$ is semisimple and $p$ is good for $G$, it follows from Theorem  \ref{T:Lparameter} and Corollary \ref{weight0} that the weights that occur in $\CL(\pi)^{ss}$ are integral and two of them differ by at least $2$.  
\vskip 5pt

Thus there are two weights $w_i$ and $w_j$ in $\Omega$ with $|w_i - w_j| \geq 2$.  This would imply that for $z \in \Gm(k)$, the Satake parameter of the unramified representation $\Pi_z$  also has at least two weights that differ by at least $2$.  But $\Pi_z$ is unitary and generic,  because $\Pi$ is cuspidal and globally generic.  The existence of two distinct integral weights is then ruled out by the main theorem of \cite{CH}.
\end{proof}

\vskip 10pt

\section{Base change and incorrigible representations}\label{BCIR}

As we mentioned in the introduction, though our ramification result in Theorem \ref{mainthm} may seem rather weak, it could in fact serve as a starting point, in conjunction with the global input of automorphic base change,  in the (long) journey towards establishing the local Langlands correspondence for a general reductive group $G$ over local function fields. In this section, we would like to elaborate on this. 
\vskip 5pt

\subsection{Tempered base change}
Let $E = k((t))$ for a finite field $k$ and let $G$ be a connected reductive group over $E$. For any finite separable field extension $F$ of $E$,
let $\CT(G,F) \subset \CA(G,F)$ denote the set of irreducible tempered (admissible) representations of $G(F)$ with coefficients in $C$, and let
$P(\CT(G,F))$ be the set of non-empty subsets of  $\CT(G,F)$.  We write $\CL^{ss}_F$ for the semisimple Langlands parametrization
of \eqref{localparam} for $\CA(G,F)$, with a subscript to indicate the base field.  

\vskip 5pt

\begin{defn} Say $G$ admits
tempered base change if, for every pair $F \subset F'$ of finite extensions of $E$, with $F'/F$ a cyclic extension of prime order, there is a map 
$$BC_{F'/F}:  \CT(G,F) \ra P(\CT(G,F'))$$
so that for any $\pi \in \CT(G,F)$ and any $\pi' \in BC_{F'/F}(\pi)$,
\begin{equation}\label{BC}
\CL^{ss}_{F'}(\pi') = \CL^{ss}_F(\pi) ~|_{W_{F'}}.
\end{equation}
\end{defn}
Of course it is expected that every $G$ admits tempered base change and that the set $BC(\pi)$ is characterized intrinsically inside $\CA(G,F)$ in
terms of $\pi$ by relations of distribution characters. It is also assumed that $BC(\pi)$ is always a finite set, but we do not make that
assumption.  There is a brief discussion of the existence of tempered base change at the end of this section.

\begin{definition}\label{incorr}   Suppose $G$ admits tempered base change.  The supercuspidal representation $\pi$ of $G(F)$ is {\bf incorrigible}
if for any sequence 
\begin{equation}\label{cyclics}
F = F_0 \subset F_1 \subset \dots \subset F_r
\end{equation}
 of extensions, where $F_i/F_{i-1}$ is cyclic of prime order for all $i \geq 1$,
the set 
$$BC_{F_r/F}(\pi) := BC_{F_r/F_{r-1}}(BC_{F_{r-1}/F_{r-2}} \dots  (BC_{F_1/F}(\pi)\dots ))$$
contains a supercuspidal member.
\end{definition}

\begin{proposition}\label{BCprop}  Assume $G$ admits tempered base change and $p \nmid |W|$.  Then for any extension $F/E$,
no pure supercuspidal representation of $G(F)$ is incorrigible.  More generally, suppose $G$ has the property that, for every Levi subgroup $L(F)$ of $G(F)$,
 every supercuspidal representation of $L(F)$
contains an $\mathfrak{s}$-type.  Then no pure supercuspidal representation of $G(F)$ is incorrigible.
\end{proposition}
\begin{proof}  Let $\pi$ be a pure supercuspidal representation of $G(F)$.  The image of the inertia group under $\CL^{ss}_F(\pi)$ is finite,
so there is a sequence of cyclic extensions of prime order, as in \eqref{cyclics}, such that, if $\pi_r$ is any member of
$BC_{F_r/F}(\pi)$, we have
$$\CL^{ss}_{F_r}(\pi_r) = \CL^{ss}_F(\pi)\mid_{W_{F_r}} $$
is unramified.  It then follows from Corollary \ref{sph} that $\pi_r$ is a constituent of an unramified principal series representation.
In particular, no member of $BC_{F_r/F}(\pi)$ can be supercuspidal.
\end{proof}

\subsection{Existence of tempered base change}\label{basec}   Suppose for the moment that $F$ is a $p$-adic field.   A procedure for defining 
cyclic base change can be constructed using the methods of \cite{Lab}.  Suppose first that $\pi$ is supercuspidal with central character of
finite order.  Then there is 
\begin{itemize}\label{items}
\item[(i)] A totally real number field $K$ with a local place $w$ such that $K_w \isoarrow F$;  
\item[(ii)] A  totally real cyclic extension $K'/K$ with $K'_w = K' \otimes_K F \isoarrow F'$, 
\item[(iii)] A connected reductive group $\CG$ over $K$, with $\CG(K_w) \isoarrow G(F)$ and 
\item[(iv)] $\CG(K_\sigma)$ compact modulo  center for all archimedean places $\sigma$ of $K$,
\item[(v)]  And an automorphic representation $\Pi$ such that $\Pi_w \isoarrow \pi$ and 
\item[(vi)] $\Pi_{v_i}$ isomorphic to a Steinberg representation at any chosen finite set of places $v_i$.  
\end{itemize}
Then it is proved in \cite{Lab}, using the stable twisted trace formula, that there is a non-empty collection $\{\Pi'\}$ of automorphic representations
of $\CG_{K'}$ such that, at every place $u$ at which $\Pi$ and $K'/K$ are unramified, $\Pi'_u$ is the unramified base change of $\Pi_u$.
We can then define $BC(\pi)$ to be the collection of $\Pi'_w$ for all such $K, K', \CG, \Pi$.  The stable twisted characters of such collections of $\Pi'_w$ are related to 
the stable character of $\pi$ but in the absence of a canonical notion of $L$-packet we omit the precise statement.

When Labesse defined his construction in \cite{Lab}, the stabilization of the  trace formula had not yet been established in either the twisted or the untwisted
setting.  Arthur reduced the stabilization to the fundamental lemma in a series of papers shortly thereafter, and a few years later the main step in the proof
of the fundamental lemma was completed by Ng\^o.  The stabilization of the twisted trace formula is contained in \cite{MW}; see also \cite{CHLN} for references
for the untwisted case.  Labesse needed hypothesis (vi) in the above list in order to work with a simplified version of the trace formula.  In principle (vi) is
no longer necessary, but the necessary consequences of the full stabilized trace formula have not yet been established (see \cite[\S CHL.IV.B]{CHLN} for an
example of the kind of work required).  Thus for practical purposes, the construction in \cite{Lab} still provides the most complete definition of local supercuspidal 
base change in general.  

Since we are only looking for semisimple Langlands parameters, we can reduce a general tempered $\pi$ to the supercuspidal 
case by means of parabolic induction.   Unfortunately, although Fargues and Scholze have defined a semisimple parametrization of all
irreducible admissible representations of $G(F)$, when $F$ is a $p$-adic field, there is no way to relate the parameter attached to
$\pi$ to that attached to the elements of the set $BC(\pi)$ defined by Labesse's construction, unless we know how to attach 
Galois parameters to the globalizations $\Pi$ as in (v).  If we could do that, then the relation between the parameters of $\Pi_u$ and $\Pi'_u$
at unramified places would suffice, by Chebotarev density, to establish the relation \eqref{BC} at the place of interest.  This reasoning has in fact
been applied for most classical groups, and for $G_2$, but it is not available in general.

Now suppose $F = k((t))$ as above.  We can certainly globalize $\pi$ as in the number field case, and then \cite{Laf18}
supplies the necessary global parametrization that is missing in the setting of number fields.  However,  the stable twisted trace formula is lacking for
function fields, so for most $G$, tempered base change is not (yet) available.  It should nevertheless be (relatively) straightforward, although time-consuming, for specialists to prove the necessary  statements, starting with the construction of the non-invariant trace formula by Labesse and Lemaire \cite{LL}, when $p$ is large relative to the group $G$. 

\subsection{Base change of large prime degree and fields of small order}\label{basechange}

Let $\pi$ be a supercuspidal representation as in Theorems \ref{mainthm7} and \ref{notunramified}.
Those theorems are stated under the hypothesis $q = |k|> 5$, which concerns the size of the residue field $k$
rather than the characteristic.  Theorem 7.2 of \cite{BFHKT} asserts that there is a constant $c(\pi)$ such that, if $m > c(\pi)$, and $F'/F$ is cyclic
of prime degree $m$ then there exists a base change $\pi'$ of $\pi$ to $F'$, in the sense that the $\CL^{ss}(\pi')$ is the restriction
of $\CL^{ss}(\pi)$ to the Weil group of $F'$.  In particular we may assume $F' = k'((t))$ is an unramified extension, where $[k':k] = m$ and therefore $|k'| = q^m$.  In particular, the hypotheses on $q$ are satisfied for $\pi'$.   Write $\pi_{m}$ instead of $\pi'$. 
Then $\CL^{ss}(\pi)$ is unramified if and only if $\CL^{ss}(\pi_m)$ is unramified.  We thus have the following alternative:
\begin{enumerate}
\item  For some $m \geq \sup(3,c(\pi))$ $\pi_m$ is  supercuspidal, and then the results of Theorems \ref{mainthm7} and \ref{notunramified} remain valid for $\pi$, without
any hypothesis on $q$;
\item Or else for all $m \geq \sup(3,c(\pi))$ $\pi_m$ is not supercuspidal.
\end{enumerate}

Although the second alternative is clearly absurd, we don't see how to exclude it by available means.  At least for the toral supercuspidals considered by
Chan and Oi, the results of \cite[\S 8]{BFHKT} show that $\pi_m$ remains supercuspidal for almost all prime $m$.  However, we have not proved that their parameters
are pure.

The above alternatives do seem to provide a route to proving that a given pure supercuspidal $\pi$ is not incorrigible, even when $|k| \leq 5$.  If we are in the first alternative,
then the residue field satisfies the hypotheses of Theorems \ref{mainthm7} and \ref{notunramified}.   We can then define what it means for $\pi$ to be incorrigible if we have
access to cyclic base change of all prime degrees -- not just those guaranteed by \cite{BFHKT} -- and argue as in the previous sections.  Under the second (absurd) alternative, we have already reduced the supercuspidality by base change
and we can then argue by induction on the size of the cuspidal support.

Too much attention should not be given to these remarks, however.  The restriction on $|k|$ is only relevant when the residue characteristic is $2, 3,$ or $5$.  These are the primes
that tend to divide the order of Weyl groups.  So it would be unnatural to try to formulate unconditional results based on these observations.

\subsection{The case of $GL(n)$}  
We again assume $F$ is a local field of positive characteristic.  We shall illustrate the strategy discussed above for $G = GL(n)$.  It has been proved by Henniart and Lemaire in \cite{HeLe} that $G = GL(n)$ admits tempered
base change;  they even prove  that $BC_{F'/F}(\pi)$ is a singleton for any $\pi$.  Moreover,  
every supercuspidal representation of every Levi subgroup of $G(F)$ contains an 
$\mathfrak{s}$-type \cite{BK}.  Thus the above discussion  applies.  Before drawing the relevant conclusion, we make the following observation.

The following lemma may be derived from the work of Bushnell and Kutzko. 

\begin{lemma}\label{Bushnell-Kutzko}  Let $\pi$ be a supercuspidal representation of $GL(n,F)$ for $F$ a local field. One of the following is true:

(a) There exists a principal hereditary order $\mathfrak A \subset M_n(F)$, a field $E \subset M_n(F)$ which is an extension of $F$ of degree $>1$, such that $E^\times$ normalizes $\mathcal A$, and, writing $\mathfrak B$ for the centralizer of $E$ in $\mathfrak A$ and $\mathfrak P$ for the maximal ideal of $\mathfrak A$, a representation $\Lambda$ of $ E^\times \cdot (\mathfrak A^\times \cap (\mathfrak B + \mathfrak P) )$ such that $\pi$ is the induced representation of $\Lambda$.

(b) There exists a representation $\Lambda$ of $ F^\times GL_n (\mathcal O_F)$ such that $\pi$ is the induced representation of $\Lambda$.

\end{lemma}

\begin{proof} By \cite[Theorem 8.4.1]{BK}, there exists a simple type $(J, \lambda)$ in $G$ such that $\pi \mid J$ contains $\lambda$. The definition \cite[Definition 5.5.10]{BK} of simple type splits into two cases, (a) and (b).

The case (b) will be easier to handle -- so much so that it is not even necessary to give the definition of this case. Instead, we note that in this case, by \cite[Theorem 8.4.1(iii)]{BK}, there exists a representation $\Lambda $ of $F^\times \mathbf  U( \mathfrak A)$, where $\mathbf  U( \mathfrak A)$ is the unit group of a hereditary order $\mathfrak A$, such that $\pi = \operatorname{Ind} (\Lambda)$.  Because $\mathfrak A$ is a heredeitary order, $\mathbf U(\mathfrak A)$ is a compact open subgroup of $GL(n, F)$ and thus, up to conjugation, is contained in the maximal compact $GL_n(\mathcal O_F)$, By first inducing $\Lambda$ from $F^\times \mathbf  U( \mathfrak A)$ to $F^\times GL_n(\mathcal O_F)$, we obtain case (b) above.

Now we turn to case (a). Let $A$ be the algebra of $n \times n$ matrices over $F$. In this case, $\mathfrak A$ is a principal $\mathcal O_F$-order in $A$ and $[\mathfrak A, n, 0, \beta]$ is a simple stratum. Let us unpack this definition. The order $\mathfrak A$ is a compact open $\mathcal O_F$-subalgebra of $A$. We write $\mathfrak P$ for the Jacobson radical of $\mathfrak A$. We let $\beta$ be an element of $A$ and take $n$ to be a positive integer. We define $E = F[\beta]$, $\mathfrak B $ the centralizer of $E$ in $\mathfrak A$, and $\mathfrak P$ the maximal ideal of $\mathfrak A$. We set \cite[Definitions 1.4.3 and 1.4.5]{BK} \[ k_0(\beta, \mathfrak A) =  \max  \{ k \in \mathbb Z \mid \{ x \in \mathfrak A \mid \beta x - x\beta \in \mathfrak P^k \} \not\subset \mathfrak B + \mathfrak P \} .\]

Then the assumption that $[\mathfrak A, n, 0, \beta]$ is a simple stratum means that $\mathfrak A$ is hereditary, $E^\times$ normalizes $\mathfrak A$, $v_{\mathfrak A}(\beta) =-n$, and $0< - k_0(\beta, \mathfrak A)$ \cite[Definition 1.5.5]{BK}..

In this context, \cite[Definition 3.1.8]{BK} defines an algebra $\mathfrak J(\beta)$ by an inductive procedure. For us, the only relevant feature of this definition is that it is contained in $\mathfrak B + \beta^{ [\frac{r+1}{2} ]}$ where  $r = -  k_0(\beta, \mathfrak A) >0$ (\cite[3.1.6]{BK}) and thus $\mathfrak J \subseteq \mathfrak B + \mathfrak P$. We can next define \cite[(3.1.14)]{BK} a group $J( \beta, \mathfrak A)$ as the intersection of $\mathfrak J$ with the group of units $\mathbf U^0 (\mathfrak A) = \mathfrak A^\times$. 

We can finally state the conclusion of \cite[Theorem 8.4.1(ii)]{BK} in this case, which is that there is a representation $\Lambda$ of $E^\times J(\beta, \mathfrak A)$ such that $\pi= \operatorname{Ind}(\Lambda)$. We will use only this structure, and will not concern ourselves with the exact construction of $\Lambda$. If $E=F$ then, because $J(\beta, \mathfrak A)$ is contained in a maximal compact subgroup, we are again in case (b), so we may assume $E \neq F$.

Since $J(\beta, \mathfrak A)$ is contained in $\mathfrak A^\times \cap (\mathfrak B + \mathfrak P)$, we can induce $\Lambda$ from $E^\times J(\beta, \mathfrak A)$ to $E^\times \cdot (\mathfrak A^\times \cap (\mathfrak B + \mathfrak P))$, and we are in case (a) above.\end{proof}

\begin{lemma}\label{GLn-pure} If $\pi$ satisfies case (b) of Lemma \ref{Bushnell-Kutzko}, then $\pi$ is pure and ramified. \end{lemma} 

\begin{proof}The determinants of elements of $GL_n(\mathcal O_F)$ must have zero valuation, and because $F^\times$ consists of scalars, the determinants of its elements must have valuation divisible by $n$. It follows that, if $\alpha$ is the character of $GL(n,F)$ defined by $\alpha(g) = e^{  \frac{ 2 \pi i}{n} v( \det (g))}$ then
\[ \pi \otimes \alpha = \operatorname{Ind} (\Lambda)\otimes \alpha=\operatorname{Ind} (\Lambda \otimes \alpha )= \operatorname{Ind} (\Lambda  ) = \pi .\]

By \cite[Remarque 0.2]{GLa}, it follows that the Genestier-Lafforgue parameter $\sigma_\pi$ is stable under tensor-product with the one-dimensional unramified representation of $W_F$ sending Frobenius to $e^{ 2\pi i/n}$. Let $a$ be a matrix representing that isomorphism, then $\sigma_{\pi} (\operatorname{Frob}_q) \circ a = e^{ 2\pi i/n} a \circ \sigma_\pi(\operatorname{Frob}_q)$ so $a$ is conjugate to $e^{2 \pi i/n }a$. Thus the set of eigenvalues of $a$ is stable under multiplication by $e^{ 2\pi i/n}$, so $a$ has $n$ distinct eigenvalues, and thus $a$ has a one-dimensional eigenspace. This one-dimensional eigenspace is a one-dimensional representation $\chi$ of the subgroup of $W_F$ corresponding to the degree $n$ unramified extension $F_n$ of $F$, which, because Frobenius permutes the eigenspaces, gives an isomorphism $\operatorname{Ind}_{ W_{F_n}}^{W_F} \chi \cong \sigma_\pi$.  Since $\chi$ is a one-dimensional representation, it is automatically pure of some weight (possibly non-integral), so $\sigma_\pi$ is pure.  By our previous result, $\sigma_\pi$ is ramified. \end{proof}

We now assume given $\pi$ in case (a) of Lemma \ref{Bushnell-Kutzko}, and prepare for the construction of a Poincair\'e series by choosing compact subgroups at each place.

It follows from an elementary calculation or  \cite[p. 76 and p. 183]{BK} that a principal hereditary order $\mathfrak A$, after choosing a basis, can be put into the following form: We choose a divisor $e$ of $n$ and then we take all $n\times n$ matrices decomposed into $n/e \times n/e$ blocks of size $e \times e$, where all the blocks on the diagonal and above have entries in the ring of integers $\mathcal O_F$, and all the the blocks below the diagonal have entries in the maximal ideal of $\mathcal O_F$. We fix for the remainder of this argument a basis where $\mathfrak A$ has this form.

We can choose a finite field $\mathbb F_q$ such that $F \cong \mathbb F_q((t))$ and fix such an isomorphism. This identifies $F$ with the local field $\mathbb F_q(t)_0$ of $\mathbb P^1_{\mathbb F_q}$ at $0$.  We now define subgroups $U_x$ of $GL_n(\mathbb F_q(t)_x)$ for each place $x$ of $\mathbb F_q(t)$. For $x\neq 0,\infty$, we take $U_x$ to be the standard maximal compact subgroup. We take $U_\infty$ to consist of block-diagonal matrices in $GL_n(\mathbb F_q((t^{-1})))$ where all the blocks on the diagonal and below are in $M_e (\mathbb F_q[[t^{-1}]])$, all the blocks above the diagonal are in $M_e ( t^{-1} \mathbb F_q[[t^{-1}]])$, and all the diagonal blocks, modulo $t^{-1}$ are in $n/e$ fixed Borel subgroups of $GL_n$, to be chosen in Lemma \ref{describe-S}. We take $U_0$ to be $ E^\times (\beta, \mathfrak A)$.

Let $ S= GL_n(\mathbb F_q(t)) \cap \prod_x U_x$.

\begin{lemma}\label{describe-S} For a suitable choice of Borel subgroups, the group $S$  is contained in a maximal torus $T$ of $GL_n(\mathbb F_q)$, and furthermore, that the $n$-dimensional vector space, when viewed as a representation of $S$, is a sum of $n/k$ characters each repeated $k$ times for some $k>1$ a divisor of $n$. \end{lemma}

\begin{proof}
First observe that the valuation of the determinant of any element of $U_x$ for any $x \neq 0$ is $0$. It follows that for $g \in S$, $\det g$ has valuation $0$ at every place away from $0$, thus by the product formula has valuation $0$ at $0$ as well.  The valuation of the determinant of every element of $\mathfrak A^\times $ is $0$, so $g$ must be the product of an element of $\mathfrak A^\times \cap (\mathfrak B + \mathfrak P)$ with a unit  $u \in E^\times$. Now every unit of $E^\times$ is contained in $\mathfrak A^\times$ by \cite[Proposition 1.2.1(i)]{BK}. Because $u \in E^\times$ and thus $u$ commutes with $\beta$, we have $u \in \mathfrak B$. Thus \[ g \in    u \cdot  (\mathfrak A^\times \cap (\mathfrak B + \mathfrak P) )  = \mathfrak A^\times \cap (\mathfrak B + \mathfrak P) .\]

In particular, $g$, at $0$, is contained in $\mathfrak A$. Together with our assumptions at other places, we conclude that every matrix entry of $g$ is integral at each place of $\mathbb F_q(t)$, hence in fact lies in $\mathbb F_q$, so $g \in GL_n(\mathbb F_q)$. Furthermore, by the integrality conditions at $0$, we see that every $e \times e$ block of $g$ below the diagonal vanishes, and by the integrality conditions at $\infty$, every $e \times e$ block of $g$ above the diagonal vanishes. So $g$ is a block-diagonal matrix, i.e. $g \in (GL_e(\mathbb F_q))^{n/e}  $. Furthermore, each of the diagonal matrices must be contained in some fixed Borel, to be chosen later.

We now use the fact that $E/F$ is a nontrivial extension. Thus it either ramifies or contains an extension of the residue field.

We handle the case when $E$ is ramified first. In this case, $E$ contains an element $x$ whose determinant does not have valuation divisible by $n$, and thus the action of $x$ on $\mathfrak A$ must permute the $n/e$ diagonal blocks.  The action on the blocks divides them into $n/(ek)$ orbits of size $k$ for some $k>1$ dividing $n/e$. Because elements of $\mathfrak B$ commute with $E$, and thus in particular with $x$, their value mod $\mathfrak p_F$ on one block in the orbit determines their value on every other block in the orbit. So, mod $\mathfrak P$, the elements of $\mathfrak B$ lie in $M_e(\mathbb F_q)^{n/(ek)}$, and the units of $\mathfrak B$ lie in $GL_e(\mathbb F_q)^{n/(ek)}$. To define $U_{\infty}$, we choose $n/e$ Borel subgroups of $GL_e(\mathbb F_q)$ such that, in each orbit, at least one of the Borel subgroups is sent by $x$ to a Borel in general position with respect to the next Borel in the orbit.  It follows that each element of $S \subset GL_e(\mathbb F_q)^{n/(ek)}$ lies in the intersection of two Borels in general position, and thus lies in the maximal torus, which is contained in a maximal torus of $GL_n$. Furthermore, the $n$-dimensional representation of $S$ is the restriction of the $n$-dimensional representation of $GL_e(\mathbb F_q)^{n/(ek)}$, which is the sum of $n/(ek)$ standard representations, each repeated $k$ times, so as a representation of $S$ it is the sum of $n/k$ characters, each repeated $k$ times, as desired.

We handle the case when $E$ contains a residue field extension $\mathbb F_{q^k}$ next. In this case, $E$ contains an element $x$ which generates the residue field extension. Because $x$ is a unit, $x$ lies in $\mathfrak A$. Restricting to each of the $n/e$ diagonal blocks and modding out by $\mathfrak p_F$, the element $x$ satisfies the characteristic polynomial of a generator of $\mathbb F_{q^k}$ and thus generates a subfield isomorphic to $\mathbb F_{q^k}$. Thus, the centralizer of $x$ in each block is isomorphic to $GL_{e/k}(\mathbb F_{q^k})$, and every unit in $\mathfrak B + \mathfrak P$, restricted to each diagonal block and reduced modulo $\mathfrak p_F$, must lie in this $GL_{e/k}(\mathbb F_{q^k})$.

We choose a Borel as follows: Fix a $\mathbb F_{q^k}$-basis of the $e$-dimensional vector space $v_1,\dots, v_{e/k}$, and then extend the sequence \[v_1,\dots, v_{e/k}, xv_{e/k}, xv_{e/k-1},\dots, xv_1\] arbitrarily to an $\mathbb F_q$-basis of this space. For each block, choose a Borel consisting of linear transformations which are upper-triangular with respect to this basis, and use these to define $U_{\infty}$. Then every $g \in S$, restricted to this block, will lie in the intersection of $GL_{e/k}( \mathbb F_{q^k})$ with this Borel. Hence $g$ will be upper-triangular with respect to the $\mathbb F_{q^k}$-basis $v_1,\dots, v_k$ and upper-triangular with respect to the $\mathbb F_{q^k}$-basis $xv_{e/k}, x v_{e/k-1},\dots, xv_1$, thus $\mathbb F_{q^k}$-diagonal with respect to the basis $v_1,\dots, v_k$. Furthermore, $g$ must act on $v_1,\dots, v_k$ by multiplication by elements of $\mathbb F_q$, since otherwise it would not preserve the $\mathbb F_q$-subspace generated by $v_1,\dots, v_k$, so in fact $g$ is contained in the diagonal of $GL_{e/k}(\mathbb F_q)$ inside $GL_{e/k}(\mathbb F_{q^k}) $ inside $GL_e(\mathbb F_q)$.  Passing from $GL_{e/k}(\mathbb F_q)$ to $GL_e(\mathbb F_q)$ in this way sends the maximal torus to a subset of the maximal torus, as desired, and expresses the $e$-dimensional standard representation of $GL_e$ as a sum of $k$ copies of the $e/k$-dimensional standard representation of $GL_{e/k}$, hence the sum of $k$ copies of $e/k$ characters of $S$, as desired. \end{proof}

Note that $U_\infty$ is an Iwahori subgroup of $GL_n ( \mathbb F_q((t^{-1})))$ and thus surjects onto $((\mathbb F_q)^\times)^n$.

\begin{lemma}\label{chi-extension} If $q>2$, any (one-dimensional) character of $S$ may be extended to a nontrivial character of $U_\infty$.

Furthermore, if $q>3$, any character of $S$ may be extended to two characters, $\chi_a,\chi_b$, both pulled back from $((\mathbb F_q)^\times)^n$, where $\chi_a$ is trivial on at least one copy of $\mathbb F_q^\times$ and $\chi_b$ is trivial on no copies of $\mathbb F_q^\times$. \end{lemma}

\begin{proof} For each case, note that the subgroup $U_{\infty}$ is an Iwahori subgroup of $GL_n ( \mathbb F_q((t^{-1})))$, so its quotient by its maximal pro-$p$ subgroup is $((\mathbb F_q)^\times)^n$.  Since $S$ is a subgroup of the maximal torus of $GL_n(\mathbb F_q)$ by Lemma \ref{describe-S}, its order is prime to $p$, and so the quotient of $U_{\infty}$ by its maximal pro-$p$ subgroup is faithful. Thus $\chi_S$ is a character of the image of $S$ inside $((\mathbb F_q)^\times)^n$.

For the first case, it suffices to extend $\chi_S^{-1}$ to a nontrivial character of $(\mathbb F_q^\times)^n$ and inflate to $U_{\infty}$. We can do this unless $\chi_S$ is trivial and $S = ((\mathbb F_q)^\times)^n$, which contradicts our claim that the characters of $S$ are repeated unless $\mathbb F_q^\times$ is trivial, which implies $q=2$.

For the second case, if $q>3$, note that since the standard representation of $S$ is the sum of $n/k$ characters repeated $k$ times, the image of $S$ inside $(\mathbb F_q^\times)^n$ is contained in $(\mathbb F_q^\times)^{n/k}$, repeated $k$ times. It follows that when we extend $\chi_S^{-1}$ to a character $\chi$ of $(\mathbb F_q^\times)^n$, which can be viewed as a tuple of $n$ characters $\chi_1,\dots, \chi_n$ of $\mathbb F_q$, we can choose any tuple of characters as long as the product of $k$ characters in each of $n/k$ orbits takes some fixed value depending on the orbit. So we can certainly choose $k-1$ of the characters in each orbit to be trivial, and the last one to take the fixed value, to produce $\chi_a$. To produce $\chi_b$, we choose $k-2$ of the characters in each orbit to be an arbitrary nontrivial character. The product of the two remaining characters is then determined. Since $q>3$, there are more than two characters, and so we can choose the next character to be a character which is neither trivial nor the determined product. It follows that the last character is nontrivial, so $\chi_b$ indeed contains no trivial characters. \end{proof}

We are now ready to construct our Poincar\'e series. Since $S$ is a finite abelian group, $\lambda$ restricted to $S \subset E^\times J(\beta, \mathfrak A)$ is a sum of one-dimensional characters. Let $\chi_S$ be one of these.

\begin{prop}\label{ramifiedGLn} Assume the residue field of $F$ has more than $2$ elements. Then every supercuspidal representation of $GL(n,F)$ is ramified. Moreover, if the residue field of $F$ has more than $3$ elements, then every supercuspidal representation of $GL(n,F)$ is either ramified and pure or wildly ramified.
\end{prop}
\begin{proof} Let $F = \mathbb F_q((t))$. Fix a supercuspidal rerpesentation $\pi$. We apply Lemma \ref{Bushnell-Kutzko}. In case (b), $\pi$ is pure and ramified by Lemma \ref{GLn-pure}, so we may reduce to case (a), where we will prove that $\pi$ is ramified if $q>2$ and wildly ramified if $q>3$.

To do this, we construct a Poincar\'e series, using all the notation above. Since $S$ is a finite abelian group, $\lambda$ restricted to $S \subset E^\times J(\beta, \mathfrak A)$ is a sum of one-dimensional characters. Let $\chi_S$ be one of these. We extend $\chi_S^{-1}$ to a nontrivial character $\chi$ of $U_\infty$ using Lemma \ref{chi-extension}.  We take the representation $\lambda \otimes \chi$ of $\prod_x U_x$, which is trivial by construction on $S =GL_n(F) \cap \prod_x U_x$. Using a matrix coefficient of this representation, we can construct a Poincar\'e series. After decomposing into Hecke eigenforms, this series generates an automorphic representation $\Pi$ which is unramified outside of $0$ and $\infty$, of type $\operatorname{Ind} (\lambda) = \pi$ at $0$, and whose local factor at $\infty$ contains a vector that transforms according to $\chi$ under $U_{\infty}$. By Lafforgue, associated to this representation is a global parameter $\sigma_\Pi$ which matches $\sigma_\pi$ at $0$, and is unramified away form $0$ and $\infty$.

Furthermore, at $\infty$, the local factor of $\Pi$ is contained in a parabolic induction from the Borel of a tamely ramified character of the maximal torus, which, restricted to the $\mathbb F_q$-points of the maximal torus, is $\chi$. Thus, by the compatibility of Genestier-Lafforgue with parabolic induction and the local Langlands correspondence for $GL_1$, the Genestier-Lafforgue parameter at $\infty$ is a sum of $n$ tamely ramified characters whose restriction to the inertia group is trivial if and only if the restriction of $\chi$ to the corresponding copy of $ \mathbb F_q^\times$ is trivial. Thus the restriction of the global parameter $\sigma_\Pi$ to $\infty$ is tamely ramified, and is unipotent if and only if $\chi$ is trivial.

If $\pi$ is unramified, then $\sigma_\Pi$ is unramified away from $0$ and $\infty$ and has at worst unipotent ramification at $0$. Because $\sigma_\Pi$ is unramified away from $0$ and $\infty$ and tamely ramified at $0$ and $\infty$, the restrictions of $\sigma_\Pi$ to the inertia groups at $0$ and $\infty$ are isomorphic, which contradicts the fact that $\sigma_\Pi$ is unipotent and $0$ and $\sigma_\Pi$ is non-unipotent at $\infty$, it is non-unipotent at $0$. So in fact $\sigma_\pi$ is ramified.

If $q>3$ and $\pi$ is tamely ramified, then we apply the second case of Lemma \ref{chi-extension} to produce two characters, $\chi_a$ and $\chi_b$, and follow the same procedure above to produce two globalizations $\Pi_a, \Pi_b$. Again $\sigma_{\Pi_a},\sigma_{\Pi_b}$ are tame at $0$ and $\infty$, so the restrictions of $\sigma_{\Pi_a}$ to the inertia subgroup at $\infty$ is isomorphic to its restriction to the inertia subgroup at $0$ and thus isomorphic, up to semisimplification, to $\sigma_\pi$. The same is true for $\sigma_{\Pi_b}$, so $\sigma_{\Pi_a}$ and $\sigma_{\Pi_b}$ are isomorphic up to semisimplification. However, the semisimplification of $\sigma_{\Pi_a}$ is a sum of one-dimensional characters, at least one trivial, whereas the semisimplification of $\sigma_{\Pi_b}$ is a sum of one-dimesional characters, all nontrivial, so they cannot be isomorphic. This is a contradiction, hence $\pi$ is wildly ramified in this case.\end{proof}

\begin{cor}\label{incorrGLn}  Let $\pi$ be a supercuspidal representation of $GL(n,F)$.  Then there is a sequence of cyclic extensions as in \eqref{cyclics}
such that $BC_{F_r/F}(\pi)$ is an irreducible constituent of an unramified principal series representation.
\end{cor}
\begin{proof}  By \ref{sph} and induction we know that we can find a sequence of cyclic extensions as in \eqref{cyclics} such that $BC_{F_r/F}(\pi)$ is a irreducible
constituent of an principal series representation whose parameter is unramified.  By Proposition \ref{ramifiedGLn}, such a principal series must be a parabolic induction from the Borel, as a supercuspidal on any Levi except the maximal torus would have a ramified Langlands parameter. By the local Langlands conjecture for $GL_1$, it must be an induction of an unramified representation of the maximal torus, i.e. an unramified principal series.
\end{proof}

Corollary \ref{incorrGLn} is  the key point in (almost) every proof of the local Langlands correspondence for $GL(n)$.  It was already mentioned in the introduction that in \cite{LRS93,HT01,He00}, this result
is obtained as a consequence of Henniart's numerical correspondence; whereas in \cite{Sch13}  it is proved by a geometric argument involving nearby cycles in the local   
model.  Starting with Corollary \ref{incorrGLn} one obtains the full local correspondence by an inductive study of the fibers of local base change.

Here we obtain Corollary \ref{incorrGLn} by a completely different argument, based on the geometric considerations of \cite{Laf18}, the global base change
of \cite{HeLe}, and the existence of types for supercuspidal representations \cite{BK}.  It is questionable whether this argument is genuinely new, because the existence of the global parametrization in \cite{Laf18}, combined
with the converse theorem as in \cite{Laf02}, already suffices to prove the local as well as the global correspondence.  However, the reasoning used here can be applied to
any $G$ for which global base change and the existence of types is known.

\section{Applications to groups over $p$-adic fields}\label{DKFS}
In this final section, we speculate on how our results in positive characteristic could have impact on analogous questions in characteristic zero local fields, via the principle of close local fields.
\vskip 5pt

Let $F = k((t))$ be a local field of characteristic $p$ as before, and let  $F^\sharp$ be a $p$-adic field.   We write $\CO_F, \CO_{F^\sharp}$ for their integer rings,
and $m_F$, $m_{F^\sharp}$ for the corresponding maximal ideals.  
\begin{defn}\label{nclose}  Let $n$ be a positive integer.  We say that $F$ and $F^\sharp$ are {\bf $n$-close} if
$$\CO_F/m_F^n \isoarrow \CO_{F^\sharp}/m_{F^\sharp}^n$$
as rings.
\end{defn}

We write $F^?$ for either $F$ or $F^\sharp$.  Let $G$ be a connected reductive group over $F^?$.  The {\it depth} of a parameter $\phi \in \CG^{ss}(G,F)$ is defined to be the maximum $n$
such that $\phi$ is trivial on the subgroup $I_{F^?}^n$ of the inertia group $I_{F^?}$, where we are using the upper numbering.
Let $\CG^{ss,n}(G,F) \subset \CG^{ss}(G,F)$ denote the subset of parameters of depth at most $n$.

\begin{thm}[Deligne]\label{deln}  If $F$ and $F^\sharp$ are $n$-close, then there is a natural bijection
\begin{equation}\label{delnn}
\CG^{ss,n}(G,F) \simeq \CG^{ss,n}(G,F^\sharp).
\end{equation}
\end{thm}

Assume now that $G$ is split, and let $I_n(F^?) \subset G(\CO_{F^?})$ denote the $n$-th Iwahori filtration subgroup, as defined in \cite[\S 3]{G}; thus $I = I_0$ is the usual Iwahori
subgroup, and $I_n$ is the kernel of the reduction map $\mathbf{I}(\CO_{F^?}) \ra \mathbf{I}(\CO_{F^?}/m_{F^?}^n).$.        
Let $H(G(F^?),I_n)$ denote the Hecke algebra of $I_n(F^?)$-biinvariant functions on $G(F^?)$ with coefficients in the algebraically closed field $C$.  
Let $\CA^n(G,F^?) \subset \CA(G,F)$ denote the subset of equivalence classes of irreducible admissible representations generated by vectors fixed under $I_n(F^?)$.
Any $\pi \in \CA^n(G,F^?)$ is then determined up to isomorphism by the representation of $H(G(F^?),I_n)$ on its invariant subspace $\pi^{I_n(F^?)}$.
The following is Ganapathy's refinement of a theorem of Kazhdan:

\begin{thm}[Ganapathy-Kazhdan, \cite{G}]\label{kazg}  If $F^\sharp$ and $F$ are $n$-close, then there is a natural isomorphism
\begin{equation}\label{hn} 
H(G(F),I_n) \isoarrow H(G(F^\sharp),I_n)
\end{equation}
of finitely-generated $C$-algebras.  Moreover, there is a bijection
\begin{equation}\label{gak} \CA^{n}(G,F) \simeq \CA^{n}(G,F^\sharp)
\end{equation}
with the property that, if $\pi \in \CA^n(G,F)$ corresponds to $\pi^\sharp \in \CA^n(G,F^\sharp)$, then the 
invariant subspaces $\pi^{I_n(F)}$ and $\pi^{\sharp,I_n(F^\sharp)}$ are isomorphic as modules with respect to the isomorphism
\eqref{hn}.
\end{thm}

In view of these results, it would be unnatural not to make the following conjecture:
\begin{conj}\label{commute}  For any positive integer $n$ the following diagram commutes:
\begin{equation*}
\begin{CD}   \CA^n(G/F)  @>Genestier-Lafforgue>>  \CG^{ss,n}(G/F) \\
 @V \eqref{gak} VV    @VV \eqref{delnn} V \\
\CA^n(G/F^\sharp)  @>>Fargues-Scholze>  \CG^{ss,n}(G/F^\sharp)
 \end{CD}
\end{equation*}
\end{conj}

Here the top line is the parametrization of Genestier and Lafforgue (Theorem \ref{paramVL} (iii)), while the bottom line is the parametrization
defined by Fargues and Scholze in \cite{FS} for groups over $p$-adic fields.  

The following corollary is an immediate consequence of the Conjecture and the results of the previous sections.

\begin{cor}\label{padic}  Let $F^\sharp$ be a finite extension of $\Qp$ with residue field $k$ of characteristic $p$.  Let $n \geq 0$ and assume there is a local field $F$ of  characteristic $p$ such that $F^\sharp$ and $F$ are $n$-close.  Let $G$ be a split semisimple group over $F$.  Suppose $|k| > 5$ and $p$ does not divide the order of the Weyl group of $G$.
Finally, suppose Conjecture \ref{commute} holds for the $p$-adic field $F^\sharp$.  

Let $\sigma$ be a supercuspidal representation of $G(F^\sharp)$.  Suppose the semisimple Fargues-Scholze parameter $\CL^{ss}(\sigma)$ attached to $\sigma$ is pure.  Then 
$\CL^{ss}(\sigma)$ is ramified.    

More generally, let $\sigma$ be any pure irreducible representation of $G(F^\sharp)$ with unramified Fargues-Scholze parameter.  Then $\sigma$ is an irreducible constituent of an
unramified principal series representation.
\end{cor} 

Since every $p$-adic field has a local field $F$ of characteristic $p$ that is $0$-close to $F^\sharp$, we have the following special case of Corollary \ref{padic}

\begin{cor}\label{padiczero}  Let $F^\sharp$ be a finite extension of $\Qp$ with residue field $k$ of characteristic $p$.  Let $G$ be a split semisimple group over $F$.  Suppose $|k| > 5$ and $p$ does not divide the order of the Weyl group of $G$.  Finally, suppose Conjecture \ref{commute} holds for the $p$-adic field $F^\sharp$.  Then any pure depth zero supercuspidal representation of $G(F^\sharp)$ has ramified Fargues-Scholze parameter.
\end{cor}

One can hope to derive consequences of Conjecture \ref{commute} for general supercuspidal representations of a general $p$-adic field $F^\sharp$, replacing $F^\sharp$ by a Galois extension $F'$ with an $n$-close $F$ for any $n$.  This is the strategy followed by Henniart in his proof of the numerical local correspondence for $p$-adic fields in \cite{He88}; the Galois extension $F'$ in his construction was obtained by applying  the results of \cite{AC} on cyclic base change.   For general groups over $p$-adic fields one can define a version of local cyclic base change using the global methods of \cite{Lab}, as explained in \S \ref{basec}.  In order to apply this to prove results about purity, we need to know the answers to the following questions:

\begin{question}  Does the Fargues-Scholze parametrization commute with cyclic base change, as described in \S \ref{basec}?
\end{question}

\begin{question}  Let $F'/F^\sharp$ be a cyclic extension of prime order.  Suppose $\pi \in \CA^n(G/F^\sharp)$, and $\pi'$ belongs to the base change of $\pi$ to $G(F')$.
Is $\pi' \in \CA^n(G/F')$?
\end{question}

\appendix
\section{Globalization of discrete series, by Rapha\"el Beuzart-Plessis}

Let $X$ be a smooth proper algebraic curve $X$ over $\mathbb{F}_q$, with function field $K = \mathbb{F}_q(X)$. Let $G$ be a connected reductive group over $K$ and let $Z^0$ be the neutral component of its center. Recall that for $v\in \lvert X\rvert$ and $\pi_v$ a discrete series of $G(K_v)$ a {\it pseudo-coefficient} of $\pi_v$ is a function $\varphi_{\pi_v}\in C_c^\infty(G(K_v)/Z^0(K_v),\xi^{-1}_v)$, where $\xi_v$ denotes the restriction of the central character of $\pi_v$ to $Z^0(K_v)$, such that for every tempered irreducible representation $\pi'_v$ of $G(K_v)$ on which $Z^0(K_v)$ acts through the character $\xi_v$, we have
$$\displaystyle \mathrm{Trace}\; \pi'_v(\varphi_{\pi_v})=\left\{\begin{array}{ll}
1 & \mbox{ if } \pi'_v\simeq \pi_v; \\
0 & \mbox{ if } \pi'_v\not\simeq \pi_v.
\end{array} \right.$$
The existence of pseudo-coefficients can be deduced from the Trace Paley-Wiener theorem \cite{BDK}. Moreover, it follows from Langlands classification that if $\pi'_v$ is an arbitrary smooth irreducible representation of $G(K_v)$, on which $Z^0(K_v)$ acts through the character $\xi_v$, if $\mathrm{Trace}\; \pi'_v(\varphi_{\pi_v})\neq 0$ then $\pi'_v$ has the same cuspidal support as $\pi_v$.

Let $\xi: Z^0(K)\backslash Z^0(\mathbf{A}_K)\to \mathbf{C}^\times$ be a continuous character.

Let $T$ be the universal Cartan of $G$ i.e. the torus quotient of a Borel subgroup in any quasi-split inner form of $G$. Let $S_{ram}\subset |X|$ be the finite subset of places where $G$ or $\xi$ is ramified and choose for every $v\in |X|\setminus S_{ram}$ a hyperspecial maximal compact subgroup $J_v\subset G(K_v)$ such that $J_v=G(\mathcal{O}_v)$ for almost all $v$. By the Satake isomorphism, ($J_v$-)unramified irreducible representations of $G(K_v)$ are parametrized by $\widehat{T}_v//W_v$ where $\widehat{T}_v$ denotes the complex torus of unramified characters of $T(K_v)$ and $W_v=Norm_{G(K_v)}(T)/T(K_v)$ the Weyl group. We denote by $(\widehat{T}_v//W_v)_{\xi_v}$ the subvariety of unramified characters $\chi\in \widehat{T}_v$ such that $\chi\mid_{Z^0(K_v)}=\xi_v$.

The purpose of this appendix is to prove the following lemma.   

\begin{lemma}  \label{L:appendix}
Let $S, S' \subset |X|$ be disjoint subsets of places of $X$ with $S'\cap S_{ram}=\emptyset$. For each $v \in S$ let $\pi_v$ be a discrete series representation whose central character restricted to $Z^0(K_v)$ equals $\xi_v$, with at least one $v_0 \in S$ such that $\pi_{v_0}$ is supercuspidal. Also, for $v\in S'$ let $X_v\subset (\widehat{T}_v//W_v)_{\xi_v}$ be a proper closed subset for the Zariski topology. Then there exists a (necessarily cuspidal) automorphic representation $\Pi \simeq \bigotimes'_{x \in |X|} \Pi_x$ of $G(\mathbf{A}_K)$, on which $Z^0(\mathbf{A})$ acts through the character $\xi$, such that:
\begin{enumerate}
\item For all $v \in S$, $\Pi_v$ has non-zero trace against a pseudo-coefficient for $\pi_v$ (in particular, $\Pi_v$ and $\pi_v$ have the same cuspidal support),
\item For all $v\in S'$, $\Pi_{v}$ is an unramified representation with Satake parameter in $(\widehat{T}_v//W_v)_{\xi_v}\setminus X_v$.
\end{enumerate}
\end{lemma}

\begin{proof}
This follows from an application of the Deligne-Kazhdan simple trace formula \cite[\S A.1 p.41]{DKV}. More precisely, we consider test functions $f=\otimes_v f_v\in C_c^\infty(G(\mathbf{A}_K)/Z^0(\mathbf{A}_K),\xi^{-1})$ satisfying:
\begin{enumerate}
\item For each $v\in S$, $f_v=\varphi_{\pi_v}$ is a pseudo-coefficient of $\pi_v$. Moreover, for $v=v_0$ we assume, as we may, that $f_{v_0}$ is a matrix coefficient of $\pi_{v_0}^\vee$;

\item For each $v\in S'$, $f_v(g)=\int_{Z^0(K_v)} f'_v(z^{-1}g) \xi_v(z)dz$ where $f'_v\in \mathcal{H}(G(K_v),J_v)$ is a $J_v$-spherical function whose Satake transform (a regular function on $\widehat{T}_v//W_v$) vanishes identically on $X_v$ but not on $(\widehat{T}_v//W_v)_{\xi_v}$;

\item For almost all $v$, $f_v$ equals the function $f^0_v$ which coincides with $zk\mapsto \xi_v(z)^{-1}$ on $Z^0(K_v)J_v$ and is zero outside;

\item For all $\delta\in G(K)$, if the $G(\mathbf{A}_K)$-conjugacy class of $\delta$ intersects the support of $f$ then $\delta$ is regular elliptic.
\end{enumerate}

Then, the Deligne-Kazhdan simple trace formula can be applied to $f$ yielding the following identity:
\begin{equation}\label{DK simple TF}
\displaystyle \sum_\Pi \mathrm{Trace}\; \Pi(f)=\sum_{\delta} \iota(\delta)^{-1}vol(G_\delta(K)Z^0(\mathbf{A}_K)\backslash G_\delta(\mathbf{A}_K))O_\delta(f)
\end{equation}
where the left sum runs over cuspidal automorphic representations $\Pi$ of $G(\mathbf{A}_K)$ with central character $\xi$, the right sum runs over elliptic regular conjugacy classes in $G(K)/Z^0(K)$, $G_\delta$ (resp. $G_\delta^+$) denotes the centralizer of $\delta$ (resp. $Z^0 \delta$) in $G$, $\iota(\delta)=[G_\delta^+(K):G_\delta(K)]$ and
$$\displaystyle O_\delta(f)=\prod_v \underbrace{\int_{G_\delta(K_v)\backslash G(K_v)} f_v(g_v^{-1}\delta g_v) dg_v}_{=: O_\delta(f_v)}$$
stands for the orbital integral of $f$ at $\delta$. Moreover, both sums are finite.

The choice of the test functions $f_v$ for $v\in S\cup S'$ implies that the only nonzero contributions to the left hand side of \eqref{DK simple TF} come from cuspidal representations satisfying the conclusion of the lemma. Thus, it suffices to see that $f$ can be arranged so that the right hand side of \eqref{DK simple TF} is nonzero. For this we need an auxiliary lemma.

\begin{lemma}
We can find local test functions $f_v$, for $v\in S\cup S'$, satisfying conditions (i) and (ii) above as well as a regular elliptic element $\delta\in G(K)$ such that $G_\delta^+=G_\delta$ and $O_\delta(f_v)\neq 0$ for every $v\in S\cup S'$.
\end{lemma}

\begin{proof}
If weak approximation holds for the inclusion $G(K)\hookrightarrow \prod_{v\in S\cup S'} G(K_v)$, then the argument is pretty standard. More precisely, we can choose any set of test functions $(f_v)_{v\in S\cup S'}$ satisfying (i) and (ii). Indeed, by definition of the Satake transform, for every $v\in S'$, the hyperbolic regular orbital integrals of $f_v$ are not identically zero  (where by {\it hyperbolic regular orbital integral} we mean one that is associated to a regular element in the maximal torus of a Borel subgroup) whereas by \cite{BP}, for each $v\in S$, the elliptic regular orbital integrals of $f_v$ are not identically zero. From this, weak approximation and the local constancy of regular semi-simple orbital integrals, we deduce the existence of $\delta\in G(K)$ that is elliptic regular in $G(K_v)$ for every $v\in S$ (hence in particular, $\delta$ is regular elliptic in $G(K)$ since $S\neq \emptyset$) and such that $O_\delta(f_v)\neq 0$ for every $v\in S\cup S'$. Moreover, we can also certainly arrange to have $G_\delta^+=G_\delta$ as the subset of regular semi-simple elements $\delta_v\in G(K_v)$ satisfying $G_{\delta_v}^+=G_{\delta_v}$ is open and dense.

To deal with the general case (that is including the case where weak approximation fails), we introduce the closure $\overline{G(K)}$ of $G(K)$ in $G(K_{S\cup S'})=\prod_{v\in S\cup S'} G(K_v)$. According to the argument at the beginning of \cite[Proof of Proposition 7.9 p.419]{PR} (which is valid regardless of the characteristic of the global field), $\overline{G(K)}$ is an open subgroup of finite index in $G(K_{S\cup S'})$. Let $v\in S\cup S'$ and let $H_v$ be a normal open subgroup of finite index of $G(K_v)$ contained in $\overline{G(K)}\cap G(K_v)$. Then, we just need to check the existence of a test function $f_v$ satisfying condition (i) or (ii) above (according to whether $v\in S$ or $v\in S'$) and a regular semi-simple element $\delta\in H_v$ such that $O_\delta(f_v)\neq 0$.  First we consider the case $v\in S'$. Then, the image of $H_v\cap T(K_v)$ in $T(K_v)/T(\mathcal{O}_v)$ is a subgroup of finite index corresponding to a finite quotient $\widehat{T}'_v$ of $\widehat{T}_v$. Then, we look for a regular function on $\widehat{T}_v//W_v$ vanishing identically on $X_v$ but not on $(\widehat{T}_v//W_v)_{\xi_v}$ and whose push forward to $\widehat{T}'_v$ via the projection $\widehat{T}_v\to \widehat{T}'_v$ is nonzero. Such a function is readily seen to exist. Consider now the case $v\in S$ and let $f_v$ be a pseudo-coefficient of $\pi_v$. By \cite{BP}, it suffices to show that the restriction of the Harish-Chandra character $\Theta_{\pi_v}$ of $\pi_v$ to the subset $H_{v,ell-rs}$ of elliptic regular semisimple relements in $H_v$ is nonzero. For $\delta$ a finite dimensional representation of $G(K_v)/H_v$, $\pi_v\otimes \delta$ is a direct sum of discrete series with Harish-Chandra character $\Theta_{\pi_v\otimes \delta}=\Theta_{\pi_v}\Theta_\delta$ where $\Theta_\delta$ denotes the usual character of $\delta$. Then, if $\Theta_{\pi_v}\mid_{H_{v,ell-rs}}=0$ the restriction of $\sum_{\delta} \frac{1}{\dim \delta}\Theta_{\pi_v\otimes \delta}$ to the elliptic regular semisimple locus of $G(K_v)$ would be zero and this would contradict the elliptic orthogonality relations of \cite{BP} on Harish-Chandra characters of discrete series.
\end{proof}

Let $(f_v)_{v\in S\cup S'}$ and $\delta\in G(K)$ as in the above lemma. Then, we choose the functions $f_v\in C_c^\infty(G(K_v)/Z^0(K_v),\xi_v^{-1})$ for $v\notin S\cup S'$ such that $O_\delta(f_v)\neq 0$ for all $v$ and $f_v=f^0_v$ for almost all $v$ (this is certainly possible thanks to the condition $G_\delta^+=G_\delta$). Then, the set of semi-simple conjugacy classes of $G(K)$ whose $G(\mathbf{A}_K)$-conjugacy classes meet the support of $f$ is finite. Hence, up to shrinking the support of $f_v$ at some auxiliary place $v\notin S\cup S'$, we may assume that this set only consists in the orbit of $\delta$. The function $f$ then satisfies condition (iii) above and the right hand side of the trace formula \eqref{DK simple TF} is reduced to the term corresponding to $\delta$. As $O_\delta(f)\neq 0$ by construction, this implies that the left hand side does not vanish either and we are done.
\end{proof}

\end{document}